\documentclass[12pt,tikz]{amsart}

\usepackage{header}

 \definecolor{bananayellow}{rgb}{.9961, .8789, .2070}
 \colorlet{polygonfill}{bananayellow!50}
 \tikzset{
   focus-focus value/.pic={
     \draw [thick] (-.1, -.1) -- (.1, .1) (-.1, .1) -- (.1, -.1);
   }
 }
 
 \newcommand\blfootnote[1]{%
  \begingroup
  \renewcommand\thefootnote{}\footnote{#1}%
  \addtocounter{footnote}{-1}%
  \endgroup
}

\renewcommand{\thefootnote}{\alph{footnote}}

\begin{document}

\title{Semitoric systems of non-simple type}
\author{Joseph Palmer \,\,\,\,  \'Alvaro Pelayo \,\,\,\, Xiudi Tang}
\date{}

\begin{abstract}
Within integrable systems, the class of so called ``semitoric'' integrable systems in dimension four has attracted a lot of attention in recent years, especially since fundamental examples from classical and quantum mechanics have been identified as semitoric by different groups of researchers. 
Several of these examples, however, show a particular trait not included in the original theory, that is, the presence of multiple (i.e.~two or more) rank zero isolated singularities in the same energy-momentum level sets. 
Systems with this property are called non-simple. 
This paper extends the original theory of Pelayo and V\~u Ng\d{o}c to non-simple systems. 
\end{abstract}




\maketitle

\blfootnote{\,\,\emph{AMS primary codes:} 53D05, 53D20, 53D35, 37J35.}

\section{Introduction} \label{sec:intro}


In Hamiltonian dynamics there is a large class of dynamical systems known as ``(Liouville) integrable'', which are those systems that have a maximal number of functionally independent quantities that are conserved under the dynamics. We refer to~\cite{AHsurvey,PVNsurvey,Pelayo21,Pel22} for recent surveys from the perspective of symplectic geometry. Integrable systems play a prominent role in many parts of physics. Within this class of systems, there is a special subclass whose elements display a rotational symmetry on one of its components and are called ``semitoric'' for this reason. 

\subsection{Semitoric systems: motivation and classification}
Semitoric integrable systems have received a lot of attention in recent years, for at least three reasons:
\begin{itemize}
\item[(i)] they are a natural generalization of toric integrable systems~\cite{De88} in dimension four;
\item[(ii)] they display a mixture of rigidity and flexibility;
\item[(iii)] they model simple yet intriguing physical systems like the Jaynes-Cummings model~\cite{JaCu63}  and the coupled angular momenta, see for example~\cite{LFP2018}.
\end{itemize}

In dimension four, the Hamiltonian flows of the integrals in a toric integrable system induce a $\mathbb{T}^2$-action,
whose fixed points are exactly the elliptic-elliptic singular points of the system. The Hamiltonian flow of the integrals
in a semitoric integrable system generates an $(S^1\times \R)$-action, whose fixed points are either elliptic-elliptic singularities or
focus-focus singularities of the system.
The existence of these focus-focus singularities are the source of the interesting and complicated behaviors of semitoric integrable systems.

About ten years ago, semitoric integrable systems were classified by Pelayo--V\~u Ng\d{o}c~\cite{PVN2009,PVN2011} under a few assumptions, in terms of five invariants. One of these assumptions was ``simplicity'', meaning each 2-dimensional singular fiber contains at most one rank zero singularity (i.e.~at most one fixed point of the associated $(S^1\times \R)$-action). In recent years new examples of semitoric systems have surfaced, of great interest in classical and quantum mechanics, which are semitoric but unfortunately do not satisfy the simplicity assumption. The goal of this paper is to extend the classification to include these systems, hence broadening the practical impact of the classification on the physics and dynamics literature, and opening the door to a better understanding of examples in these areas. 

\subsection{Main result and example}
The main result of this paper is Theorem~\ref{thm:equiv-classification}, which extends the classification
of simple semitoric systems due to Pelayo--V\~u Ng\d{o}c~\cite{PVN2009, PVN2011} to also include non-simple semitoric systems by adapting the invariants.
Theorem~\ref{thm:equiv-classification} classifies semitoric systems, simple or not, in terms of a convex polygon decorated with marked points and certain labels on these points.
The proof of this theorem entails following the proofs of the original classification and carefully monitoring the effect of
the presence of multipinched fibers.
While the majority of the proof can be adapted (and hence why our paper is quite short), there is a subtlety in those parts where the twisting index plays a role, and which can already be seen when analyzing the behavior of this index for specific examples.
The construction of the analogue of the twisting index in the non-simple case is mixed with the construction of the Taylor series invariants, and so they must be packaged together in the classification.


In this paper we 
state and prove two equivalent versions of 
the complete classification for the non-simple case, which takes into account this more complicated version of the twisting index.
Firstly, we introduce an object called \emph{the complete semitoric invariant}, which packages together all of the invariants of a semitoric system into a single object. We show that 
\emph{the complete semitoric invariant classifies semitoric systems, simple
or not, up to isomorphisms;} 
this is the content of Theorem~\ref{thm:classification}, which we state and prove in Section~\ref{sec:classification} once we have introduced the necessary ingredients for the formulation.
Secondly, we introduce an object called a \emph{marked labeled semitoric polygon} and, making use of Theorem~\ref{thm:classification}, we prove Theorem~\ref{thm:equiv-classification}, which states that \emph{marked labeled semitoric polygons classify semitoric systems, simple
or not, up to isomorphisms}.

These classifications are equivalent, but we believe that both versions are of interest. The complete semitoric invariant 
shows some information in a clearer way (and introduces the idea of a \emph{wall-crossing index} in this context), but  the marked labeled semitoric polygon holds equivalent information and is somewhat easier to work with.
Furthermore, Theorem~\ref{thm:equiv-classification} is more similar to the original classification of simple systems.
Note that the polygons involved in the classification stated in Theorem~\ref{thm:classification} are not necessarily convex, but the polygons involved in the classification stated in Theorem~\ref{thm:equiv-classification} are always convex.

\subsection{Non-simple semitoric systems in physics and symplectic topology}
Symplectic geometry, although a mathematical subject on its own right, has its origins in classical mechanics and much of its development has been motivated by problems in classical mechanics, see Abraham--Marsden~\cite{AbMa}, De Le\'on--Rodrigues~\cite{dLRo}, Libermann--Marle~\cite{LiMa} and Marsden--Ratiu~\cite{MaRa} for books which emphasize this aspect of the subject.

In recent years several authors have discovered examples of non-simple semitoric 
systems, such as the system in~\cite{HohPal} which generalizes the coupled angular momentum 
system (see Section~\ref{sec:examples}), and the system in~\cite{DeMeulenaere-Hohloch} which, for
certain values of the parameters, has two focus-focus fibers which are each a double-pinched torus
containing two focus-focus points (so the system has four focus-focus points in total).

To be more concrete, following Hohloch--Palmer~\cite{HohPal}, the generalized coupled angular-momentum system on $\mathbb{S}^2\times \mathbb{S}^2$ with Hamiltonian \[J=z_1+z_2\] and first integral \[H=\frac{z_1+z_2+2x_1y_1+2x_2y_2}{4}\] has as energy-momentum level set $(J,H)^{-1}(0,0)$ a double pinched torus, depicted in Figure~\ref{fig:intro-HP}. This is a typical non-simple semitoric system,  and we will examine this example more closely in Section~\ref{sec:examples}.

\begin{figure}[ht]
  \centering
  \inputfigure{intro-HP}
  \caption{For certain values of the parameters in the generalized coupled angular momenta system from~\cite{HohPal} a double-pinched torus appears.}
  \label{fig:intro-HP}
\end{figure}

Furthermore, the system discussed in~\cite[Section 6]{Dawson-Dullin-Nguyen} contains a fiber which includes
two focus-focus points and in~\cite[Theorem 6]{Dawson-Dullin-Nguyen} the authors prove
that it is a so-called \emph{proper semitoric system}.
Proper semitoric systems are integrable systems $F=(J,H)\colon M\to\R^2$ which satisfy all conditions
to be semitoric except for the fact that while $F$ is proper, $J$ may not be.
Such systems were studied in~\cite{PRVN2017}, and examples
include the spherical pendulum.
The example from~\cite{Dawson-Dullin-Nguyen} provides further evidence that non-simple semitoric systems (in this case a 
non-simple proper semitoric system) appear in mechanics,
and for additional examples of semitoric systems in physics, see for instance~\cite{JaCu63,SaZi99,BD2015}.

Moreover, the study of singular foliations is crucial 
within the various forms of symplectic dynamics, see~\cite[Section 2]{BH-firststeps}.
Multipinched tori appear in mirror symmetry~\cite{GW1997} (we thank Mark Gross for discussions).  To get Lagrangian torus fibrations in mirror symmetry one can start with a K3 surface with an elliptic fibration and by hyperk\"ahler rotation turn it into a special Lagrangian fibration, the singular fibers of which can include the multipinched tori.

Toric fibrations with singularities, such as those that arise from semitoric systems, are important in the context
of mirror symmetry and algebraic geometry~\cite{GrSi2003, GrSi2006, GrSi2010, GrSi2011}, and symplectic geometry~\cite{LeSy2010, Zung1996, RaWaZu2018}.
Also, in~\cite{Vianna1,Vianna2,Vianna3} Vianna uses almost-toric fibrations, nodal trades, and nodal slides (as in~\cite{Sym}) to construct infinitely many non-Hamiltonian isotopic Lagrangian tori in $\C\mathrm{P}^2$ and monotone del Pezzo surfaces.

One could also consider questions related to the Hamiltonian displaceability or non-displaceability of fibers that arise in semitoric systems. 
For instance, in~\cite[page 8]{ElPo2010} the authors study the displaceability properties
of the regular (i.e.~torus) fibers in a specific semitoric system. Non-displaceability of a given fiber is related
to whether or not the fiber has unobstructed Floer cohomology, and recently there has been progress in
defining Floer theories for immersed Lagrangians submanifolds~\cite{AJ,PW1,PW2} which can have
self-intersection points. The pinched tori that show up in semitoric systems are examples of such immersed Lagrangians,
and thus, one could also begin study of displaceability properties of the focus-focus fibers (single-pinched or multi-pinched)
that arise in semitoric systems.

\subsection{Can one hear non-simple semitoric systems?}

The classification and invariants of non-simple semitoric systems which we present in this paper should be very useful in understanding whether one can or cannot ``hear'' non-simple semitoric systems (and to what extent, in the affirmative case), in the sense of being able or not being able to recover the classical system (of principal symbols) from the semiclassical spectrum of a quantum semitoric system. The term ``hear" in this context was popularized by  Kac’s famous article \emph{``Can you hear the shape of a drum?”} \cite{KAC}.

In other words, one would like to know whether the semiclassical spectrum completely determines or not the (isomorphism class) of the symplectic manifold and the semitoric system on it, that is, whether one can hear the semitoric system. The Pelayo--V\~{u} Ng\d{o}c conjecture (Pelayo--V\~{u} Ng\d{o}c \cite[Conjecture 9.1]{PVNsteps}) states that the answer is yes for simple systems, that is, \emph{one can hear simple semitoric systems}. 

The conjecture has now been shown by the work of the following authors: Pelayo--V\~{u} Ng\d{o}c~\cite{PVN14} and Le Floch--Pelayo--V\~{u} Ng\d{o}c~\cite[Theorem A]{LFPVN2016} proved that one can hear the first four invariants of simple semitoric systems, and recently Le Floch--V\~{u} Ng\d{o}c~\cite{LFVN} proved that one can hear the fifth invariant, the twisting index. Since by the work of Pelayo--V\~{u} Ng\d{o}c semitoric systems are classified by these five invariants, the conjecture follows. In addition Le Floch--V\~{u} Ng\d{o}c~\cite{LFVN} gave algorithms to recover the invariants.

The following question remains: when the semitoric system is not simple, can we still hear it? And if not always, then in which cases? The answers are unknown, but we expect the symplectic invariants of non-simple semitoric systems that we define in this paper are important to understand this problem. We refer to Section~\ref{sec:quantum} for further details and references.

In~\cite{Pel22} the second author proposes a collection of open problems concerning classical and quantum integrable systems, including an existence problem, sometimes also referred to as a ``surjectivity statement" for quantum semitoric systems~\cite[Problem 12.5]{Pel22}.
Non\--simple semitoric systems are likely to give, at a quantum level, sources of examples of semitoric systems which have equivalent semiclassical spectra despite being non-isomorphic at the classical level (see Section~\ref{sec:recover-from-affine} for a related discussion). Several open problems in this direction are suggested by the second author in~\cite[Sections 12.7 and 12.9]{Pel22}.

\subsection{Structure of the paper}

In Section~\ref{sec:invariants} we describe the required background and describe how to construct a set of invariants from a given semitoric system, simple or not.

In Section~\ref{sec:definition} we define the \emph{complete semitoric invariant} and prove that the map assigning such an invariant to a semitoric system is well-defined up to isomorphism (Proposition~\ref{prop:welldef}).

In Section~\ref{sec:classification} we prove that the map assigning the complete semitoric invariant to a semitoric system is injective (Proposition~\ref{prop:uniqueness}) and surjective (Proposition~\ref{prop:existence}).
In Section~\ref{sec:classification} we also state and prove Theorem~\ref{thm:classification}.

In Section~\ref{sec:equiv-classification} we explain an equivalent version of this classification which is more consistent with the original classification of semitoric systems, and making use of Theorem~\ref{thm:classification} we state and prove the main result of the paper, Theorem~\ref{thm:equiv-classification}.

In Section~\ref{sec:twist} we discuss the behavior of the twisting index invariant of a non-simple semitoric system.

In Section~\ref{sec:examples} we explicitly discuss an example of a non-simple semitoric system which is a special case of the class of systems studied in~\cite{HohPal}.

In Section~\ref{sec:final} we give final remarks and discuss various open problems related to non-simple semitoric systems.
 
\section{Preliminaries} \label{sec:invariants}

Let $(M, \om)$ be a connected symplectic $4$-manifold.
Integrable systems $(M, \om, F \colon M \to \R^2)$ with non-degenerate singularities can have critical points of six types: elliptic-regular, hyperbolic-regular, elliptic-elliptic, elliptic-hyperbolic, hyperbolic-hyperbolic, and focus-focus.
In this paper we do not consider systems which have singularities of elliptic-hyperbolic, hyperbolic-hyperbolic, or hyperbolic-regular type.
A critical point $p$ of $F$ is a \emph{focus-focus point} if there are local coordinates $(x_1,\xi_1,x_2,\xi_2)$ 
and a local diffeomorphism $g\colon \R^2\to\R^2$ such that $\om = \der x_1 \wedge\der \xi_1 + \der x_2 \wedge \der \xi_2$, $p = (0,0,0,0)$, and
\begin{equation} \label{eq:FFnormal}
  g\circ F = (x_1\xi_2-x_2\xi_1, x_1 \xi_1 +x_2 \xi_2).
\end{equation}
The local models for elliptic-regular and elliptic-elliptic points are, respectively, 
\[
 g \circ F = \left( \frac{x_1^2 + \xi_1^2}{2}, \xi_2 \right) \quad \textrm{ and } \quad g \circ F = \left( \frac{x_1^2 +\xi_1^2}{2}, \frac{x_2^2 + \xi_2^2}{2}\right).
\]
A \emph{focus-focus fiber} is any fiber of $F$ which contains at least one focus-focus point and all critical points in the fiber are
focus-focus points; topologically these fibers are $2$-tori pinched once for each focus-focus point.


A rich class of systems $(M,\om,F=(J,H))$ having these types of singularities are those called \emph{semitoric}, which means that $J$ is a proper function whose Hamiltonian flow is $2\pi$-periodic and $F$ has only non-degenerate singularities of these three types, elliptic-regular, elliptic-elliptic, and focus-focus (i.e.~no hyperbolic components).  
\begin{definition}
 Two semitoric systems $(M_i,\om_i,F_i = (J_i,H_i))$, $i \in \{1,2\}$, are
 \emph{isomorphic} if there exists a symplectomorphism $\phi \colon M_1 \to M_2$ such that $\phi^*(J_2,H_2) = (J_1, f(J_1,H_1))$ for some smooth function $f(x,y)$ such that $\frac{\partial f}{\partial y}>0$ everywhere. 
\end{definition}
\begin{definition}
 A semitoric system is \emph{simple} if each fiber of $J$ (and hence also each fiber of $F$) contains at most one focus-focus point.
 \end{definition}
 So if $F$ is simple then the focus-focus fibers of $F$ are homeomorphic to once-pinched tori as in $F^{-1}(c_1)$ in Figure~\ref{fig:fibers}.

\subsection{Polygons via cutting at focus-focus points} \label{ssec:semitoricpolygon}
Let $(M, \om, F = (J,H))$ be a semitoric system and let $B= F(M)$.

\begin{definition}
Let $M_{\mathrm{f}} \subset M$ be the set of \emph{focus-focus points} of $F$, that is, the critical points of $F$ with local model given by expression~\eqref{eq:FFnormal}.
A \emph{focus-focus value} is an element of $F(M_{\mathrm{f}})$. 
\end{definition}

By~\cite[Corollary 5.10]{VN2007} $M_{\mathrm{f}}$ is finite.
Let $m_{\mathrm{f}}, \vf, \lamf \in \Z_{\geq 0}$ respectively be the cardinalities of $M_{\mathrm{f}}$, $F(M_{\mathrm{f}})$, and $J(M_{\mathrm{f}})$.
Let $B_{\mathrm{r}} \subset B$ denote the set of regular values of $F$.

Arrange the focus-focus values lexicographically by $(x, y) < (z, t)$ if and only if $x < z$ or both $x = z$ and $y < t$, and label them $\{c_1,\ldots,c_\vf\}=F(M_\mathrm{f})$.
For $i \in \Seq{1, \dotsc, m_{\mathrm{f}}}$, let $m_i$ be the number of focus-focus points in $F^{-1}(c_i)$, known as the \emph{multiplicity} of the focus-focus value $c_i$, so $\sum_{i = 1}^\vf m_i = m_{\mathrm{f}}$, see Figure~\ref{fig:fibers}.
By~\cite{VN2007}, the fibers of $F$ are connected, so we may associate $F(M)$ with the base of the Lagrangian fibration
on $M$ induced by $F$, which, by a result of Duistermaat~\cite{Duis1980}, inherits the structure of an integral affine manifold with corners and nodes in the following way.

\begin{definition}
For $b\in B$ each $\beta\in T^*_b\R^2$ determines a vector field $\mathcal{X}_\beta$ on $F^{-1}(b)\subset M$ 
via the equation $F^*\beta = -\omega(\mathcal{X}_\beta,\cdot)$,
and we let $\beta\in T^*_b\R^2$ act on $F^{-1}(b)$ by flowing along $\mathcal{X}_\beta$ for time $1$.
For $b \in B$ let $2\pi \Lambda_b \subset T^*_b\R^2$ be the isotropy subgroup of the action and let $\Lambda = \coprod_{b \in B} \Lambda_b$, then $(B,\Lambda)$ is an integral affine manifold with corners and nodes, as in~\cite{Duis1980}.
We call $(B,\Lambda)$ the \emph{base of $(M,\om,F)$}.
\end{definition}
Let $\Periodcan$ be the usual integral affine structure on $\R^2$. 
The focus-focus values create monodromy in the integral affine structure of $(B,\Lambda)$ and obstruct any global affine map 
$(B,\Lambda) \to (\R^2,\Periodcan)$, but we can define a map which is affine when restricted to each vertical region between the focus-focus values as follows.

\begin{proposition}\label{prop:affinecoords}
Let $(M,\omega,F=(J,H))$ be a semitoric system. Arrange the
elements of the set  $J(M_{\rm f})=\{j_1,\ldots,j_{\lambda_{\rm f}}\}$
in strictly increasing order, that is, $j_1<\ldots<j_\lamf$. 
Define $j_0=-\infty$ and $j_{\lambda_{\rm f}+1}=+\infty$. For each $a \in \{0,\ldots,\lambda_{\rm f}\}$ let $I_a=\{(x,y) \in \mathbb{R}^2 \,|\,
j_a<x<j_{a+1}\}$ and let $B_a=B_{\rm r}\cap I_a$. Then there exists
an injective, orientation preserving, continuous function $\mathcal{A} \colon B \to \mathbb{R}^2$ with the following properties:
\begin{enumerate}[label={\textup{(\arabic*)}}]
\item $\mathcal{A}$ preserves the $x$\--coordinate;
\item for every $a \in \{0,\ldots,\lambda_{\rm f}\}$ the restriction $\mathcal{A}|_{B_a}$ of $\mathcal{A}$ to $B_a$ is smooth;
\item $(\mathcal{A}|_{B_a})^*\Lambda_{\rm can}=\Lambda$, where $\Lambda_{\rm can}$ is the usual affine structure of $\mathbb{R}^2$.
\end{enumerate}
\end{proposition}

\begin{proof}
For each $a\in\{0,\ldots,\lamf\}$, by~\cite[Theorem 3.4]{VN2007} $B_a$ is a simply connected subset of $B_{\mathrm{r}}$. Thus,
since the Hamiltonian
flow of $J$ is $2\pi$--periodic there exists an
orientation preserving map $\mathcal{A}_a \colon B_a \to \R^2$ which preserves the first coordinate and for which $\mathcal{A}_a^* \Periodcan = \Lambda$, unique up to vertical translation and composition with powers of 
\begin{equation}\label{eqn:T}
 T = \begin{pmatrix} 1 & 0\\1 & 1 \end{pmatrix}.
\end{equation}
Exactly as in~\cite[Theorem 3.8]{VN2007}, after vertically translating if needed the maps $\mathcal{A}_a$, $a\in\{0,\ldots,\lambda_\mathrm{f}\}$, can be piecewise combined to a single continuous map $\mathcal{A} \colon B \to \R^2$ which is the desired map.
\end{proof}

\begin{definition}\label{def:affinecoords}
  An injective, orientation preserving, continuous function $\mathcal{A} \colon B \to \R^2$ is  a \emph{choice of piecewise affine coordinates} if it
  satisfies properties (1)--(3) from Proposition~\ref{prop:affinecoords}.
\end{definition}

Also note if $\mathcal{A}$ is a choice of piecewise affine coordinates then it preserves the lexicographic order.
  Let $\proj_i \colon \R^2 \to \R$, $i \in \{1, 2\}$ be the projection onto the $i^{\mathrm{th}}$ component and
  for $j \in \R$ let $\ell_j = \proj_1^{-1}(j)$. 

\begin{definition} \label{def:VPIA-group}
  For a finite set $\mathbf{j} \subset \R$  let $G_\mathbf{j}$ denote the \emph{vertical piecewise integral affine group}, that is, the group of homeomorphisms $\rho \colon \R^2 \to \R^2$ which preserve the first component and for which $(\rho|_{\R^2 \setminus \bigcup_{j \in \mathbf{j}}\ell_j})^*\Periodcan = \Periodcan$.
\end{definition}

The following two results are now immediate.

\begin{proposition}\label{prop:generators}
For $x\in\R$ denote by $\Heaviside_x$ the function which is $1$ when $x \geq 0$ and $0$ otherwise, and for every $j,b \in \R$ define the homeomorphisms $\mathsf{t}_j$ and $\shift_b$ of $\R^2$ by 
\begin{equation}\label{eqn:t_and_shift}
 \mathsf{t}_j(x, y) = (x, y + (x-j)\Heaviside_{(x-j)}), \qquad \shift_b(x, y) = (x, y + b).
\end{equation}
Then, $G_\mathbf{j}$ is the Abelian group generated by $T$, $\Seq{\mathsf{t}_j}_{j \in \mathbf{j}}$, and $\{\shift_b\}_{b\in\R}$, and is thus canonically isomorphic to $\Z^{\lambda_{\mathrm{f}} + 1} \times \R$, where $\lambda_{\mathrm{f}}$ is the cardinality of $\mathbf{j}$.
\end{proposition}

\begin{lemma}
  Let $(M, \om, F)$ be a semitoric system with base $(B,\Lambda)$ and let $\mathbf{j} = (j_1, \dotsc, j_\lamf)$ be the $\lamf$-tuple of images of the focus-focus points of $F$ under $J$.
  Then there is a choice of piecewise affine coordinates on $(B, \Lambda)$ which is unique up to left composition by an element of $G_\mathbf{j}$.
\end{lemma}

Let $\mathcal{A}$ be a choice of piecewise affine coordinates 
and, similarly to~\cite{VN2007, PVN2009},
let $\Delta = \mathcal{A}(B) \in \mathrm{Polyg}(\R^2)$, as in Figure~\ref{fig:complete-invariant},
where $\mathrm{Polyg}(\R^2)$ is the set of closed polygons in $\R^2$. 
For $i\in\{ 1,\ldots, \vf\}$, let
$\tilde{c}_i = \mathcal{A}(c_i)$.

For $a\in \{1, \ldots, \lamf\}$, let $s_a \in \Z_{>0}$ be the number of focus-focus values in the line $\ell_{j_a}$ and let $r_a \in \Z_{>0}$ be the lowest index such that $c_{r_a} \in \ell_{j_a}$.
Then $\Seq{\tilde{c}_{r_a}, \dotsc, \tilde{c}_{r_a+s_a-1}} = (\mathcal{A}\circ F(M_{\mathrm{f}})) \cap \ell_{j_a}$, and write $\tilde{c}_i = (\tilde{c}_i^1, \tilde{c}_i^2)$ for the components of $\tilde{c}_i$.
The vertical line $\ell_{j_a}$ is separated into $(s_a + 1)$ segments by the images under $\mathcal{A}$ of the focus-focus values,
\[
  \ell_{j_a}^\alpha = \Set{(x,y) \in \ell_{j_a} \mmid \tilde{c}_{r_a+\alpha-1}^2 < y < \tilde{c}_{r_a+\alpha}^2},
\]
for $\alpha \in \Seq{0, \dotsc, s_a}$, above taking $\tilde{c}_{r_a-1}^2 = -\infty$ and $\tilde{c}_{r_a+s_a}^2 = \infty$, so that
$
 \ell_{j_a} = \bigcup_{\alpha = 0}^{s_a} \overline{\ell_{j_a}^\alpha},
$ 
see Figure~\ref{fig:wall-crossing}.
Using the exact same argument as~\cite[Theorem 3.8]{VN2007}, we have the following.
\begin{proposition}
For each $a\in\{1,\ldots,\lambda_{\mathrm{f}}\}$ and $\alpha\in\{0,\ldots,s_a\}$ there exists a $w_a^\alpha\in\Z$ such that
\begin{equation} \label{eq:limitT}
  \lim_{\substack{x \to j_a\\ x<j_a}} \der \mathcal{A}(x, y) = T^{w_a^\alpha} \circ \lim_{\substack{x \to j_a\\ x>j_a}} \der \mathcal{A}(x, y),
\end{equation}
for any $y$ such that $\mathcal{A} (j_a, y) \in \ell_{j_a}^\alpha$.
\end{proposition}

\begin{figure}[htb]
  \centering
  \inputfigure{wall-crossing}
  \caption{A neighborhood in $\De$ of a single line $\ell_{j_a}$ which includes three marked points, so $s_a = 3$.
  The line $\ell_{j_a}$ is separated into four parts $\ell_{j_a}^0$, $\ell_{j_a}^1$, $\ell_{j_a}^2$, and $\ell_{j_a}^3$ by the marked points, and each is labeled to the left by the wall-crossing index.
  Note that the multiplicity labels on the marked points determine the difference in wall-crossing index above and below them.}
  \label{fig:wall-crossing}
\end{figure}

\begin{definition}\label{def:wallcrossing}
  We call $w_a^\alpha$ the \emph{wall-crossing index of the line segment
  $\ell_{j_a}^\alpha$} and we call $w_a^{0},\ldots,w_a^{s_a}$ the
  \emph{wall-crossing indices of the line $\ell_{j_a}$}. Furthermore, we
  call $w_a = w_a^0 \in \Z$ the \emph{lower wall-crossing index} associated to the line $\ell_{j_a}$ since it describes the wall-crossing along the lowest segment of $\ell_{j_a}$.
\end{definition}

Due to the monodromy effect of focus-focus values on the affine structure of $B_{\mathrm{r}}$, these integers are subject to 
$
 w_a^{\alpha+1} - w_a^\alpha = m_{r_a+\alpha}.
$
Hence the lower wall-crossing index $w_a$ and the multiplicities of the focus-focus fibers (i.e.~the number of focus-focus points in the fiber) determine the tuple $(w_a^0, \dotsc, w_a^{s_a}) \in \Z^{s_a+1}$.
This is illustrated in Figure~\ref{fig:wall-crossing}.
Let $\mathrm{Vert}(\R^2) = \Set{\ell_j \mmid j \in \R}$.

\paragraph{Conclusion:} For each choice of piecewise affine coordinates $\mathcal{A}$ as in Definition~\ref{def:affinecoords} we have obtained:
\begin{itemize}[nosep] 
 \item the polygon $\De \in \mathrm{Polyg}(\R^2)$ endowed with the lines $\ell_{j_1}, \dotsc, \ell_{j_{\lamf}} \in \mathrm{Vert}(\R^2)$, each labeled with a lower wall-crossing index $w_1, \dotsc, w_\lamf \in \Z$ as in Definition~\ref{def:wallcrossing}, and
 \item the points $\tilde{c}_1, \dotsc, \tilde{c}_\vf \in \mathrm{int}(\De) \cap \left(\bigcup_{a = 1}^\lamf\ell_{j_a}\right)$ which are the images under $\mathcal{A}$ of $\{c_1, \dotsc, c_\vf \} =  F(M_{\mathrm{f}})$, each labeled with a multiplicity $m_1,\ldots,m_{v_\mathrm{f}} \in \Z_{>0}$ which is the number of focus-focus points in the corresponding fiber.
\end{itemize}
See Figures~\ref{fig:complete-invariant} and~\ref{fig:fibers}.

\begin{remark}
The polygon $\De$ generalizes the notion of weighted polygon
of complexity $\lamf$ in \cite[Definition 4.4]{PVN2009},
which applied to simple systems.
Furthermore, the \emph{cartographic map} introduced in~\cite{PRVN2017} is a 
special case of a choice of affine coordinates as in Definition~\ref{def:affinecoords}.
\end{remark}

\begin{remark}
The idea of cutting the base used above was applied in~\cite{Sym} to almost-toric systems, and in~\cite{VN2007,PVN2009, PVN2011,PRVN2017,LFP2018} to semitoric systems.
A study of the case of semitoric systems on log-symplectic/b-symplectic manifolds
has been recently  initiated in~\cite{BHMM}.
\end{remark}

\subsection{Taylor expansions at focus-focus points} \label{sec:taylorseries}

Consider a focus-focus value $c_i \in F(M_\mathrm{f})$ and let $(p^i_\mu)_{\mu \in \Z_{m_i}} \subset F^{-1}(c_i)$ be the tuple of focus-focus points in the fiber over $c_i$ for each $i\in\{1,\ldots,v_\mathrm{f}\}$.
Choose the points to be in order according to the direction of the flow of $H$, so the choice of numbering is unique up to cyclic permutation, which is why we take the index $\mu$ to be in $\Z_{m_i}= \Z/m_i\Z$ as in~\cite{PT}.

By Eliasson's linearization theorem for non-degenerate focus-focus points~\cite{eliasson1984hamiltonian, MR3098203}, near $p^i_\mu$ and $c_i$ there is an orientation preserving symplectomorphism 
\[
 \varphi^i_\mu \colon (M, \omega, p^i_\mu) \to (\R^4, \omega_0, 0)
\]
and an orientation preserving diffeomorphism $E^i_\mu \colon (B, c_i) \to (\R^2, 0)$ such that $q \circ \varphi^i_\mu = E^i_\mu \circ F$ (see Figure~\ref{fig:diagram}), where 
\[
 q(x_1,\xi_1,x_2,\xi_2) = ( x_1\xi_2-x_2\xi_1, x_1\xi_1 + x_2\xi_2)
\] 
is the local model of a focus-focus point in $\R^4$, as in Equation~\eqref{eq:FFnormal}.
Furthermore, since the Hamiltonian flow of $J$ is $2\pi$-periodic we may assume that $E^i_\mu$ acts on the 
first component by translation, i.e.~that
\begin{equation} \label{eq:formE}
  \proj_1\circ E^i_\mu (x,y) = x - j_a.
\end{equation}

\begin{figure}
\centering
\[ \begin{tikzcd}
       M \supseteq &[-44pt] V_\mu^i \arrow{r}{\phi_\mu^i} \arrow[d,"F"]  &  \R^4 \arrow[d,"q"] &[-30pt] & \\
       B \supseteq \,\,\,\,\null & U_i\setminus\ell_{j_a}\arrow{r}{E^i_\mu} \arrow{d}{\mathcal{A}^2} & \R^2\cong\C  & \supseteq \, \C\setminus\mathrm{i}\R^+ \arrow[r,"K_+"] & \R\\
       & \R &   &  &  
\end{tikzcd}
\] 
\caption{A commutative diagram of the relevant maps for defining the Taylor series invariants, where
$V_\mu^i$ is a neighborhood of $p_\mu^i$. 
The action Taylor series 
$\tilde{\mathsf{s}}^i_\mu$ essentially encodes the difference between $\mathcal{A}^2$ and $K_+\circ E^i_\mu$, while
the transition Taylor series $\mathsf{g}^i_{\mu, \nu}$ encodes the difference between the second components of $E_\mu^i$
and $E_\nu^i$ (Eliasson diffeomorphisms for different focus-focus points in the same fiber).}
\label{fig:diagram}
\end{figure}

Let $\mathcal{A} = (\mathcal{A}^1,\mathcal{A}^2)\colon B\to\R^2$ be a choice of piecewise affine coordinates
as in Definition~\ref{def:affinecoords}.
In order to find invariants of $(M,\om,F)$ that are well-defined up to isomorphisms, we compare the coordinates $\mathcal{A}$ and $E^i_\mu$ in $U_i \setminus \ell_{j_a}$, where $U_i\subset B$ is a neighborhood of $c_i$ and $c_i \in \ell_{j_a}$.
We may assume that $U_i \setminus \ell_{j_a}$ has two connected components, corresponding to $x > j_a$ and $x < j_a$, and that $c_i$ is the only focus-focus value in $U_i$.

Let $\log_+ \colon \C \setminus \imag \R^+ \to \C$ be the determination of $\log$ with $\log_+ 1 = 0$ and branch cut at $\imag \R^+$.
Then, for $c \in \C \setminus \imag \R$, let $K_+ \colon \C \setminus \imag \R^+ \to \R$ be given by
\begin{equation}\label{eqn:K}
  K_+(c)= -\Im (c \log_+ c - c).
\end{equation}

Identifying $\R^2\cong\C$ so $E^i_\mu \colon B \to \C$, let $\tilde{S}^i \colon U_i \setminus \ell_{j_a} \to \R$ be given by
\begin{equation} \label{eq:def-S}
  \tilde{S}^i = 2\pi \mathcal{A}^2 - \sum_{\mu \in \Z_{m_i}} (E^i_\mu)^* K_+ + 2\pi w_a^\alpha (j_a-x)\Heaviside_{(j_a-x)},
\end{equation}
where $\Heaviside_x$ is defined as in Proposition~\ref{prop:generators}.
In the set $x>j_a$ the function $\tilde{S}^i$ takes the difference of the piecewise affine coordinates and the sum of the pull-backs of the function $K_+$. When
$x\leq j_a$ the third term accounts for how the piecewise affine coordinates change passing through $\ell_{j_a}$.

The following lemma explains how the wall-crossing indices and the functions $\tilde{S}^i$ (as in Equation~\ref{eq:def-S})
change when changing the choice of affine coordinates.

\begin{lemma} \label{lem:Gaction}
  Suppose that $\mathcal{A},\mathcal{A}'$ are two choices of piecewise affine coordinates
  as in Definition~\ref{def:affinecoords}, so there exists some $\rho \in G_\mathbf{j}$ such that $\mathcal{A}' = \rho \circ \mathcal{A}$,
  where $G_\mathbf{j}$ is as in Definition~\ref{def:VPIA-group}.
  Let $w_a^\alpha$ denote the wall-crossing indices relative to $\mathcal{A}$ as in Definition~\ref{def:wallcrossing}, let $\tilde{S}^i$ be the function given in Equation~\eqref{eq:def-S}, and
  let $(w_a^\alpha)'$ and $(\tilde{S}^i)'$ be those relative to $\mathcal{A}'$.
  Then the values of $\mathcal{A}'$, $(w_a^\alpha)'$, and $(\tilde{S}^i)'$ in terms of $\mathcal{A}$, $w_a^\alpha$, and $\tilde{S}^i$ are as follows when $\rho$ is a generator of $G_\mathbf{j}$:
  \begin{center} \begin{tabular} {|c|c|c|c|}
    \hline
    $\rho \in G_\mathbf{j}$ & $\mathcal{A}' = \rho \circ \mathcal{A}$                             & $(w_a^\alpha)'$ & $(\tilde{S}^i)'$ \rule{0pt}{2.4ex} \\ \hline \hline
    $T$ & $(\mathcal{A}^1, \mathcal{A}^2 + x)$ & $w_a^\alpha$ & $\tilde{S}^i + 2\pi x$ \rule{0pt}{2.4ex} \\ [.05em] \hline
    $\mathsf{t}_j$, $j\neq j_a$ & $(\mathcal{A}^1, \mathcal{A}^2 + (x-j)\Heaviside_{(x-j)})$     & $w_a^\alpha$ & $\tilde{S}^i + 2\pi (x-j)\Heaviside_{(x-j)}$ \rule{0pt}{2.4ex} \\ [.05em] \hline
    $\mathsf{t}_{j_a}$ & $(\mathcal{A}^1, \mathcal{A}^2 + (x-j_a)\Heaviside_{(x-j_a)})$ & $w_a^\alpha-1$  & $ \tilde{S}^i + 2\pi (x-j_a)$ \rule{0pt}{2.4ex} \\ [.05em] \hline
    $\shift_b$ & $(\mathcal{A}^1, \mathcal{A}^2 + b)$ & $w_a^\alpha$ & $\tilde{S}^i + 2\pi b$ \rule{0pt}{2.4ex} \\ [.05em] \hline
  \end{tabular} \end{center}
\end{lemma}

\begin{proof}
  The first column follows from the definitions of $T$, $\mathsf{t}_j$, and $\shift_b$ in Equations~\eqref{eqn:T} and~\eqref{eqn:t_and_shift} and the fact that $\mathcal{A}^1 = x$.
  The second column follows from Equation~\eqref{eq:limitT}, and the last column follows from the first two and Equation~\eqref{eq:def-S}, since $K_+$ and $E^i_\mu$ do not depend on the choice of $\mathcal{A}$.
\end{proof}

\begin{lemma} \label{lem:skA-smooth}
  $\tilde{S}^i$ can be extended to a smooth function in a neighborhood of $c_i$ in $B$.
\end{lemma}

\begin{proof}
  The lemma is true when $w_a^\alpha = 0$ by~\cite[Lemma 3.5]{PT}.
  Since the action of $G_\mathbf{j}$ described in Lemma~\ref{lem:Gaction} does not affect the smoothness of $\tilde{S}^i$ around $c_i$, if there is a choice of $\mathcal{A}$ such that $w_a^\alpha = 0$, then the proof is complete.
  Such a choice always exists by acting by an integer power of $\mathsf{t}_{j_a}$, which changes
  $w_a^\alpha$ by one.
\end{proof}

\begin{definition} \label{def:Taylor_series}
We still use $\tilde{S}^i$ to denote the smooth extension to a neighborhood of $c_i$.
Let $X = \der x, Y = \der y$.   Performing a Taylor expansion of $\tilde{S}^i$ around the origin under coordinates $E^i_\mu$, we get a power series
  \begin{equation} \label{eq:deftildes}
    \tilde{\mathsf{s}}^i_\mu = \Tl_0[\tilde{S}^i\circ (E^i_\mu)^{-1}] = \sum_{p, q = 0}^ \infty (\tilde{\mathsf{s}}^i_\mu)^{(p, q)} X^p Y^q,
  \end{equation}
  the \emph{action Taylor series at $p^i_\mu$}, where $\mu \in \Z_{m_i}$.
  Expanding the transition maps between coordinates $E^i_\mu$ and $E^i_\nu$, we get
  \begin{equation} \label{eq:transitionTaylor}
    \mathsf{g}^i_{\mu, \nu} = \Tl_0[\proj_2 \circ E^i_\nu \circ (E^i_\mu)^{-1}] = \sum_{p, q = 0}^ \infty (\mathsf{g}^i_{\mu, \nu})^{(p, q)} X^p Y^q,
  \end{equation}
  the \emph{transition Taylor series} from $p^i_\mu$ to $p^i_\nu$, where $\mu, \nu \in \Z_{m_i}$.
\end{definition}

From~\cite[Equation (3.3)]{PT}, we have the following.

\begin{lemma}\label{lem:invariant-constraint}
The series in Equations~\eqref{eq:deftildes} and~\eqref{eq:transitionTaylor} are constrained by the following relations:
\begin{equation} \label{eq:invariant-constraint}
\begin{cases}
  (\mathsf{g}^i_{\mu, \nu})^{(0, 1)} > 0, \\
  \tilde{\mathsf{s}}_\mu^i(X, Y) = \tilde{\mathsf{s}}_\nu^i(X, \mathsf{g}_{\mu, \nu}^i(X, Y)), \\
  \mathsf{g}_{\mu, \mu}^i(X, Y) = Y, \\
  \mathsf{g}_{\mu, \sigma}^i(X, Y) = \mathsf{g}_{\nu, \sigma}^i(X, \mathsf{g}_{\mu, \nu}^i(X, Y)),
\end{cases}
\end{equation}
for $\mu, \nu, \sigma \in \Z_{m_i}$.
\end{lemma}

Let $\R[[X,Y]]$ denote the set of Taylor series in two variables with real coefficients and $\R_+[[X,Y]]$ be the subset
of those which have zero constant terms and positive coefficients for the linear terms in $Y$.
In~\cite[Theorem 3.9]{PT} it was shown that the semi-local model in a neighborhood of $F^{-1}(c_i)$ is determined up to semi-local isomorphisms by a set $(\mathsf{s}^i_\mu, \mathsf{g}^i_{\mu, \nu})_{\mu, \nu \in \Z_{m_i}}$, where $\mathsf{s}^i_\mu \in \R[[X,Y]]/(2\pi X\Z)$ and $\mathsf{g}^i_{\mu, \nu} \in \R_+[[X,Y]]$, and conversely that the semi-local model determines $(\mathsf{s}^i_\mu, \mathsf{g}^i_{\mu, \nu})_{\mu, \nu \in \Z_{m_i}}$ up to cyclic reordering of the indices.
By \emph{semi-local} we mean in a neighborhood of the fiber (some authors also use the term \emph{semi-global} for this).
Because the construction used above is analogous to the one in~\cite[Definitions 3.6 and 3.7]{PT}, $(\mathsf{s}^i_\mu,\mathsf{g}^i_{\mu, \nu})_{\mu, \nu \in \Z_{m_i}}$ can be obtained from the invariant in Definition~\ref{def:Taylor_series} via
\begin{equation} \label{eq:tilde-to-notilde}
  (\mathsf{s}^i_\mu, \mathsf{g}^i_{\mu, \nu}) = \left(\tilde{\mathsf{s}}^i_\mu - (\tilde{\mathsf{s}}^i_\mu)^{(0, 0)} + 2\pi X \Z, \mathsf{g}^i_{\mu, \nu}\right)
\end{equation}
for each $\mu, \nu \in \Z_{m_i}$.

\begin{remark}\label{rmk:generatingset}
Because of the relations in~\eqref{eq:invariant-constraint}, the elements of the tuple $(\tilde{\mathsf{s}}_{\mu}^i,\mathsf{g}^i_{\mu,\nu})_{\mu,\nu\in\Z_{m_i}}$ are not independent. In particular, $\tilde{\mathsf{s}}_{0}^i$ and $(\mathsf{g}_{\mu,\mu+1}^i)_{\mu\in \Z_{m_i}\setminus\{m_i-1\}}$ completely
  determine the entire tuple $(\tilde{\mathsf{s}}_{\mu}^i,\mathsf{g}^i_{\mu,\nu})_{\mu,\nu\in\Z_{m_i}}$. 
  That is, given any choice of 
  $\tilde{\mathsf{s}}_{0}^i\in\R [[X,Y]]$ and $\mathsf{g}_{0,1}^i, \ldots, \mathsf{g}_{m_i-2,m_i-1}^i\in \R_+ [[X,Y]]$,
  there is exactly one possible choice of tuple $(\tilde{\mathsf{s}}_{\mu}^i,\mathsf{g}^i_{\mu,\nu})_{\mu,\nu\in\Z_{m_i}}$ which satisfies
  the system of equations~\eqref{eq:invariant-constraint}.
\end{remark}

Since the transition Taylor series only depends on the diffeomorphisms from the
Eliasson linearization theorem, and not on $\tilde{S}^i$ or $\mathcal{A}$, changing the choice of piecewise affine coordinates $\mathcal{A}$ preserves $\mathsf{g}^i_{\mu, \nu}$.
Define an action of $G_\mathbf{j}$ on $\tilde{\mathsf{s}}^i_\mu$
by specifying the action of the generators as
\begin{equation} \label{eq:new_tilde_s}
  \rho(\tilde{\mathsf{s}}^i_\mu) = \begin{cases}
    \tilde{\mathsf{s}}^i_\mu + 2\pi X + 2\pi j_a, &\textrm{if } \rho = T, \\
    \tilde{\mathsf{s}}^i_\mu + 2\pi X + 2\pi (j_{a'} - j_a), &\textrm{if } \rho = \mathsf{t}_{j_{a'}} \textrm{ and }a'\leq a, \\
    \tilde{\mathsf{s}}^i_\mu, &\textrm{if } \rho = \mathsf{t}_{j_{a'}} \textrm{ and }a'>a, \\
    \tilde{\mathsf{s}}^i_\mu + 2\pi b, &\textrm{if } \rho = \shift_b, 
  \end{cases}
\end{equation}
where $c_i \in \ell_{j_a}$.

\begin{lemma} \label{lem:Gaction-Taylor}
  Let $\tilde{\mathsf{s}}^i_\mu$ be the action Taylor series at the focus-focus point $p_\mu^i$ as in Definition~\ref{def:Taylor_series} relative to a choice of piecewise affine coordinates $\mathcal{A}$, and let $(\tilde{\mathsf{s}}^i_\mu)'$ be the action Taylor series of $p_\mu^i$ relative to a choice of piecewise affine coordinates $\mathcal{A}' = \rho\circ \mathcal{A}$ where $\rho \in G_\mathbf{j}$.
  Then $(\tilde{\mathsf{s}}^i_\mu)' = \rho(\tilde{\mathsf{s}}^i_\mu)$.
\end{lemma}

\begin{proof}
Lemma~\ref{lem:Gaction} explains how changing piecewise affine
coordinates changes the function $\tilde{S}^i$ (as in Equation~\eqref{eq:def-S}).
Combining this with Equations~\eqref{eq:formE} and~\eqref{eq:deftildes},
which show how to obtain $\tilde{s}_\mu^i$ from $\tilde{S}^i$, describes how changing affine
coordinates affects $\tilde{s}_\mu^i$, which is exactly the same as the 
definition of $\rho(\tilde{s}_\mu^i)$ from Equation~\eqref{eq:new_tilde_s}.
\end{proof}

The tuple of Taylor series invariants we have constructed depends on the choice of ordering for the focus-focus points
in the given focus-focus fiber, which is only unique up to cyclic permutation. 
Let $[\tilde{\mathsf{s}}_{\mu}^i,\mathsf{g}^i_{\mu,\nu}]_{\mu,\nu\in\Z_{m_i}}$ denote the orbit of $(\tilde{\mathsf{s}}_{\mu}^i,\mathsf{g}^i_{\mu,\nu})_{\mu,\nu\in\Z_{m_i}}$ under the action of $\Z_{m_i}$ by $z\cdot(\tilde{\mathsf{s}}_{\mu}^i,\mathsf{g}^i_{\mu,\nu})_{\mu,\nu\in\Z_{m_i}} = (\tilde{\mathsf{s}}_{\mu+z}^i,\mathsf{g}^i_{\mu+z,\nu+z})_{\mu,\nu\in\Z_{m_i}}$ for $z\in\Z_{m_i}$, where the addition in the indices is modulo $m_i$.
Note that one element of the orbit $[\tilde{\mathsf{s}}_{\mu}^i,\mathsf{g}^i_{\mu,\nu}]_{\mu,\nu\in\Z_{m_i}}$
satisfies the system of equations~\eqref{eq:invariant-constraint} if and only if all elements of the orbit satisfy those equations.

\paragraph{Conclusion:}
We have assigned to each critical value $c_i$ a tuple of Taylor series $[\tilde{\mathsf{s}}^i_\mu, \mathsf{g}^i_{\mu, \nu}]_{\mu, \nu \in \Z_{m_i}}$
with $\tilde{\mathsf{s}}^i_\mu\in\R[[X,Y]]$ and $\mathsf{g}_{\mu,\nu}^i\in\R_+[[X,Y]]$ for each $i\in\{1,\ldots,v_\mathrm{f}\}$ and $\mu,\nu\in\Z_{m_i}$.
Moreover, the $\mathsf{g}_{\mu,\nu}$ are independent of the choice of piecewise affine coordinates and when changing the choice of piecewise affine
coordinates the $\tilde{\mathsf{s}}_\mu^i$ change according to Equation~\eqref{eq:new_tilde_s} and Lemma~\ref{lem:Gaction-Taylor}.

\begin{remark} \label{rmk:S-from-tilde}
  In light of Equations~\eqref{eq:tilde-to-notilde}--\eqref{eq:new_tilde_s} and Lemma~\ref{lem:Gaction-Taylor}, note that the choice of $\mathcal{A}$ does not affect the part of $\tilde{\mathsf{s}}^i_\mu$ which represents the series from~\cite{VN2003,PT}, as expected.
\end{remark}

\begin{remark}
  The twisting index invariant (the fifth invariant in the Pelayo--V\~{u} Ng\d{o}c classification~\cite{PVN2009,PVN2011}) does not appear as an independent piece of the complete semitoric invariant, since the data of the twisting index invariant is now encoded in the $X$ coefficient of the action Taylor series $\tilde{\mathsf{s}}^i_\mu$. For a discussion of the relationship between the twisting index invariant and the complete semitoric invariant see Section~\ref{sec:twist}.
\end{remark}

\section{The complete semitoric invariant} \label{sec:definition}

In the previous section we constructed a $5$-tuple
\begin{equation} \label{eq:object}
  \tilde{\mathrm{i}}(M, \om, F) = \left(\De, (\ell_{j_a})_{a = 1}^\lamf, (w_a)_{a = 1}^\lamf, (\tilde{c}_i)_{i = 1}^\vf, \left(m_i, [\tilde{\mathsf{s}}^i_\mu, \mathsf{g}^i_{\mu, \nu}]_{\mu, \nu \in \Z_{m_i}}\right)_{i = 1}^\vf\right)
\end{equation}
starting from the system $(M, \om, F)$, which depends on the choice of piecewise affine coordinates $\mathcal{A}$ as in Definition~\ref{def:affinecoords}, so
it is not yet a symplectic invariant of $(M,\om,F)$.

In order to define the complete semitoric invariant, of simple or non-simple systems, we start with the following definition, motivated from Section~\ref{sec:invariants}.

\begin{definition}
  Let
  \begin{equation*}
    \mathfrak{T} = \Set{(m, [\tilde{\mathsf{s}}_\mu, \mathsf{g}_{\mu, \nu}]_{\mu, \nu \in \Z_m})\mmid m \in \Z_{>0}\textrm{ and }\tilde{\mathsf{s}}_\mu \in \R[[X,Y]], \mathsf{g}_{\mu, \nu} \in \R_+[[X,Y]] \textrm{ for all }\mu, \nu \in \Z_m}
  \end{equation*}
  and for $\lamf,\vf \in \Z_{\geq 0}$ let
  \begin{equation} \label{eqn:defX}
    \mathbf{X}_{\lamf, \vf} = \mathrm{Polyg}(\R^2) \times (\mathrm{Vert}(\R^2))^\lamf \times \Z^{\lamf} \times (\R^2)^\vf \times (\mathfrak{T})^\vf.
  \end{equation}
  Let $\mathbf{z} = (z_0, \dotsc,z_\lamf) \in \Z^{\lamf+1}$ and $b \in \R$.
  Let $T$, $\mathsf{t}_{j}$, and $\shift_b$ be
  as given in Equations~\eqref{eqn:T} and~\eqref{eqn:t_and_shift} and let 
  these operators act on $\tilde{\mathsf{s}}^i_\mu$ as in Equation~\eqref{eq:new_tilde_s}.
  By $\mathsf{t}_{j_a}^{z_a}$ we simply mean the composition of 
 $\mathsf{t}_{j_a}$ with itself $z_a$ times. Define the action $(\Z^{\lamf+1}\times\R)\times \mathbf{X}_{\lamf, \vf} \to \mathbf{X}_{\lamf, \vf}$ of $\Z^{\lamf+1} \times \R$ on $\mathbf{X}_{\lamf, \vf}$ by
  \begin{equation} \label{eq:Gj-action}
    (\mathbf{z},b)\cdot \left( \begin{array}{c} \De,\\[.3em]
     (\ell_{j_a})_{a = 1}^\lamf,\\[.3em]
     (w_a)_{a = 1}^\lamf,\\[.3em]
     (\tilde{c}_i)_{i = 1}^\vf,\\[.3em]
     \left( m_i, [\tilde{\mathsf{s}}^i_\mu, \mathsf{g}^i_{\mu, \nu}]_{\mu, \nu \in \Z_{m_i}}\right)_{i = 1}^\vf \end{array}\right)
    = 
    \left( \begin{array}{c}
      \shift_b\circ \mathsf{t}_{j_1}^{z_1}\circ\ldots\circ \mathsf{t}_{j_\lamf}^{z_\lamf}\circ T^{z_0}(\De),\\[.3em] 
      (\ell_{j_a})_{a = 1}^\lamf,\\[.3em]
      (w_a - z_a)_{a = 1}^\lamf,\\[.3em]
      \left( \shift_b\circ \mathsf{t}_{j_1}^{z_1}\circ\ldots\circ \mathsf{t}_{j_\lamf}^{z_\lamf} \circ T^{z_0}(\tilde{c}_i) \right)_{i = 1}^\vf,\\[.3em]
      \left( m_i, \left[\shift_b\circ\mathsf{t}_{j_1}^{z_1}\circ\ldots\circ \mathsf{t}_{j_\lamf}^{z_\lamf}\circ T^{z_0}(\tilde{\mathsf{s}}^i_\mu),\mathsf{g}^i_{\mu, \nu}\right]_{\mu, \nu \in \Z_{m_i}}\right)_{i = 1}^\vf
    \end{array}
    \right).
  \end{equation}    
\end{definition}

It is straightforward to check that Equation~\eqref{eq:Gj-action} actually
defines a group action.
The construction of $\tilde{\mathrm{i}}(M, \om, F)$ in Equation~\eqref{eq:object} is unique up to the choice of piecewise affine coordinates $\mathcal{A}$ as in Definition~\ref{def:affinecoords}, which is unique up to the action of $G_\mathbf{j}$
as in Definition~\ref{def:VPIA-group}, which is isomorphic to $\Z^{\lamf+1} \times \R$.
Since we have taken the quotient by precisely this symmetry (see Lemmas~\ref{lem:Gaction} and~\ref{lem:Gaction-Taylor}) we have:

\begin{proposition} \label{prop:welldef}
  The assignment of $(M, \om, F)\mapsto (\Z^{\lamf+1} \times \R)\cdot \tilde{\mathrm{i}}(M, \om, F)$ is a well-defined function
  which has as its domain the set of all semitoric systems $\mathcal{M}$ and as its codomain the quotient space 
  $\coprod_{\lambda_\mathrm{f},v_\mathrm{f}\in\Z_{\geq 0}} \left( \mathbf{X}_{\lamf, \vf} / (\Z^{\lambda_\mathrm{f}+1}\times\R)\right)$.
\end{proposition}

\begin{figure}[htb]
  \centering
  \inputfigure{complete-invariant}
  \caption{A representative of the complete semitoric invariant of Definition~\ref{def:complete} with $\lamf = 3$ and $\vf = 4$.
  Each marked point $\tilde{c}_i$ is indicated with an $ \times $ and each vertical line $\ell_{j_a}$ is indicated with a dashed line.
  The integral lattice $\Z^2\subset\R^2$ is also shown.
  In this example the vertices of the polygon are all on lattice points, but this is not true in general.
  The marked points are each labeled with their multiplicity $m_i$ and each segment of each vertical line is marked with its wall-crossing index to the left.
  The lower wall-crossing indices for this example are $w_1 = -2$, $w_2 = 1$, and $w_3=-4$.
  Not shown is the Taylor series label $[\tilde{\mathsf{s}}^i_\mu, \mathsf{g}^i_{\mu, \nu}]_{\mu, \nu \in \Z_{m_i}}$ on each marked point for $i \in\{ 1,2,3,4\}$.}
  \label{fig:complete-invariant}
\end{figure}

\begin{definition} \label{def:complete}
  The \emph{complete semitoric invariant of $(M, \om, F)$} is the $(\Z^{\lamf+1} \times \R)$-orbit of $\tilde{\mathrm{i}}(M, \om, F)$ from Equation~\eqref{eq:object}, see Figure~\ref{fig:complete-invariant}.
\end{definition}

\begin{remark}
 Definition~\ref{def:complete} generalizes to non-simple semitoric systems
 the Pelayo--V\~{u} Ng\d{o}c invariants as given
 by~\cite[Definition 6.1]{PVN2009}. We discuss this in Remark~\ref{rmk:recover}.
\end{remark}

The following is an immediate consequence of the definitions.
\begin{lemma} \label{lem:isom}
  Let $(M_1, \om_1, F_1)$ and $(M_2,\om_2,F_2)$ be semitoric systems.
  If $(M_1, \om_1, F_1)$ and $(M_2,\om_2,F_2)$ are isomorphic then they have the same complete semitoric invariant.
\end{lemma}

\begin{remark}\label{rmk:quotientmap}
 Lemma~\ref{lem:isom} is equivalent to the fact that the function $(M,\om,F)\mapsto (\Z^{\lamf+1}\times\R)\cdot\tilde{\mathrm{i}}(M,\om,F)$ induces
 a well-defined function $\mathrm{i}$ given by
 \begin{align}
  \label{eqn:i} \mathrm{i}\colon\mathcal{M}  /{\sim}  & \to \coprod_{\lambda_\mathrm{f},v_\mathrm{f}\in\Z_{\geq 0}} \left( \mathbf{X}_{\lamf, \vf} / (\Z^{\lambda_\mathrm{f}+1}\times\R)\right)\\ \nonumber
  [(M,\om,F)] &\mapsto (\Z^{\lamf+1}\times\R)\cdot \tilde{\mathrm{i}}(M,\om,F),
 \end{align}
 where $[(M,\om,F)]$ denotes the isomorphism class of $(M,\om,F)$.
\end{remark}

\section{Classification}
\label{sec:classification}

In this section we explain how to remove
the simplicity assumption in the classification of semitoric systems
in~\cite{PVN2009, PVN2011}, leading us to a classification 
which applies to both the simple and the non-simple cases, formulated in terms of the complete semitoric
invariant of Definition~\ref{def:complete}.

\subsection{Uniqueness} \label{ssec:uniqueness}

We will use the following result from~\cite{PT}, which is a generalization of~\cite[Lemma 5.1]{VN2003}, in an essential way in the upcoming proof of Proposition~\ref{prop:uniqueness}.

\begin{lemma}[{\cite[Lemma 4.1]{PT}}]
\label{lem:PT-localdiff}
Let $E_j$, $j \in \Z_k$, be diffeomorphisms from a neighborhood $U$ of $0$ in $\R^2$ to a neighborhood of $0$ in $\R^2$ sending $0$ to itself.
We denote $\sigma = \sum_{j \in Z_k}E_j^* \der K_+$, where $K_+$ is as given in Equation~\eqref{eqn:K}.
Let $\tau_1, \tau_2 \in \Omega^1(U)$ such that $\tau_2 - \tau_1$ is closed and flat.
Then there is a diffeomorphism $G \colon U \to G(U) \subset \R^2$ such that $G^*(\tau_2 + \sigma) = \tau_1 + \sigma$.
\end{lemma}

\begin{proposition} \label{prop:uniqueness}
  Let $(M_1, \om_1, F_1)$ and $(M_2,\om_2,F_2)$ be semitoric systems.
  Then $(M_1, \om_1, F_1)$ and $(M_2,\om_2,F_2)$ are isomorphic if and only if they have the same complete semitoric invariant as in Definition~\ref{def:complete}.
\end{proposition}

\begin{proof}
  The implication from left to right is Lemma~\ref{lem:isom}.
  For the implication from right to left we follow~\cite[pages 588--596]{PVN2009} and only prove statements which are not
  already present in that proof.
  The proof in~\cite{PVN2009} is split into three steps, we start with \emph{Step 1: first reduction.}
  Assume that $(M_1, \om_1, F_1)$ and $(M_2, \om_2, F_2)$ are semitoric systems which have the same complete semitoric invariant:
  \[
    (\Z^{\lamf+1}\times\R) \cdot \left( \De, (\ell_{j_a})_{a = 1}^\lamf, (w_a)_{a = 1}^\lamf, (\tilde{c}_i)_{i = 1}^\vf, \left( m_i, [\tilde{\mathsf{s}}^i_\mu, \mathsf{g}^i_{\mu, \nu}]_{\mu, \nu \in \Z_{m_i}}\right)_{i = 1}^\vf \right)\in \mathbf{X}_{\lamf, \vf} / (\Z^{\lambda_\mathrm{f}+1}\times\R).
  \]
  
 Recall that different representatives of the complete semitoric invariant
 correspond to different choices of piecewise affine coordinates as in
 Definition~\ref{def:affinecoords}.
  First, we choose the same representative of the complete semitoric invariant for each system, so in particular 
  as in the conclusion of Section~\ref{ssec:semitoricpolygon} they have the same polygons $\De$ and the same ordered
  tuple of lower wall-crossing indices (as in Definition~\ref{def:wallcrossing}).
  This means that for $i\in\{1,2\}$ there exist piecewise affine coordinates 
  as in Definition~\ref{def:affinecoords} for each system which have the same image $\De$, we
  denote them by $\mathcal{A}_i = (\mathcal{A}_i^1,\mathcal{A}_i^2) \colon F_i(M_i) \to \De$
  (in~\cite{PVN2009} $\mathcal{A}_i$ is denoted instead by $g_i^{-1}$, but in this paper we use $\mathsf{g}_{\mu, \nu}^i$ in the Taylor series following~\cite{PT}).

  Define $h = \mathcal{A}_1^{-1}\circ \mathcal{A}_2$. We wish to replace $F_2$ by $\tilde{F}_2 = F_2 \circ h$ so that $\mathrm{Image}(\tilde{F}_2) = \mathrm{Image}(F_1)$.
  In order for $\tilde{F}_2$ to be semitoric and isomorphic to $F_2$ the crucial  point is to show that $h(x,y) = (x, f(x,y))$ for some smooth function $f$.
  By~\cite[Theorem 3.8]{VN2007} $h$ has this form but \emph{a priori} $f$ is not smooth.
  If the systems are simple an argument in~\cite{PVN2009} shows that the fact that $F_1$ and $F_2$ have the same invariants (there are five invariants~\cite[Definition 6.1]{PVN2009}) implies that $h$ is smooth~\cite[Claim 7.1]{PVN2009}.
  The argument is unchanged away from the focus-focus values, so the proof of $h$ being a diffeomorphism can be referred to~\cite{PVN2009} except for the smoothness near a focus-focus value $c_i$, which we explain next.

  By the fact that the transition Taylor series of Equation~\eqref{eq:transitionTaylor} are the same for the two systems, the local diffeomorphisms $(E^i_\nu)_1$ and $(E^i_\nu)_2$ (from Section~\ref{sec:taylorseries}) can be chosen to be equal
  up to a transformation which is infinitely tangent to the identity and leaves the first variable unchanged.
  In part 2 of the proof of~\cite[Lemma 5.1]{VN2003}, it is shown that any such local diffeomorphism of $\R^2$ with those properties
  can be lifted to an automorphism of the local model of the focus-focus point, and thus by composing with that automorphism
  we may assume that $(E^i_\nu)_2 = (E^i_\nu)_1$.

  Considering Equation~\eqref{eq:deftildes}, the fact that the action Taylor series are the same for the two systems, and the fact that $(E^i_\nu)_2 = (E^i_\nu)_1$, we may invoke~\cite[Lemma 2.18]{PT} to conclude that $\mathcal{A}_2^2$ and $\mathcal{A}_1^2$ are equal up to a flat function.
 To justify the smoothness of $h$ near the focus-focus point $c_i$, we apply Lemma~\ref{lem:PT-localdiff} taking $\tau_1 = \mathrm{d}\mathcal{A}_1 - \sigma$ and $\tau_2 = \mathrm{d}\mathcal{A}_2 - \sigma$.
 Since $\tau_2-\tau_1 = \der (\mathcal{A}_2 - \mathcal{A}_1)$ is closed and flat, Lemma~\ref{lem:PT-localdiff} implies that there is a local diffeomorphism $G$ from a neighborhood $U$ of $c_i$ in $F_1(M_1)$ to a neighborhood of the corresponding focus-focus value in $F_2(M_2)$ such that $G^*(\tau_2 + \sigma) = \tau_1 + \sigma$ and thus
    \begin{equation}\label{eqn:EdA2} 
   G^*\der \mathcal{A}_2 = \der \mathcal{A}_1
  \end{equation} 
in $U \setminus \ell_i$, where $\ell_i$ is the vertical line through $c_i$ in $F_1(M_1)$.
  Thus, Equation~\eqref{eqn:EdA2} implies that $G^*\mathcal{A}_2 - \mathcal{A}_1$ is locally constant in $U\setminus \ell_i$, and since the function $G^*\mathcal{A}_2 - \mathcal{A}_1$ is continuous across the vertical line $\ell_i$, it is constant in $U$.    
  Since
  \begin{equation*}
    \lim_{(x,y)\to c_i}(G^*\mathcal{A}_2 - \mathcal{A}_1)(x,y)=\mathcal{A}_2(G(c_i)) - \mathcal{A}_1(c_i)=0;
  \end{equation*}
  this implies that $G^*\mathcal{A}_2 - \mathcal{A}_1=0$ in $U$.
  Thus we obtain that $G^*\mathcal{A}_2 = \mathcal{A}_1$ and therefore $h = G^{-1}$, so $h$ is smooth in a neighborhood of the focus-focus value $c_i$, as desired; see Figure~\ref{fig:smoothness}.
  This completes \emph{Step 1} from~\cite{PVN2009}
  (in~\cite{PVN2009} the authors also discuss the necessity that the two systems have equal twisting index, which is not something we need to
  consider in our case since that information is now encoded in the new Taylor series).
  Thus we may, and do, assume that $F_2(M_2)=F_1(M_1)$.

  \begin{figure}[ht]
    \centering
    \inputfigure{smoothness}
    \caption{A diagram of the maps involved in the argument that $h$ is smooth in the proof of Proposition~\ref{prop:uniqueness}.}
    \label{fig:smoothness}
  \end{figure}

  In \emph{Step 2} of~\cite{PVN2009} it is proven that the semitoric systems $F_1$ and $F_2$ can be intertwined by symplectomorphisms using~\cite[Theorem 2.1]{VN2003} on the preimages $F^{-1}_i(\Omega_{\alpha})$, $i\in\{1,2\}$, $\alpha \in I$, where the collection of sets $\Omega_{\alpha}$ is a convenient covering of the common base $F_1(M_1) = F_2(M_2)$.
  The sets of the covering are defined in such a way that they are of four types: 1) contain no critical points of the $F_i$, 2) contain critical points of rank $1$ but not rank $0$, 3) contain a critical point of rank $0$ of elliptic type; 4) contain a critical point of rank $0$ of focus-focus type.
  In our case the construction of the covering $\Seq{\Omega_{\alpha}}_{\alpha \in I}$ is identical to~\cite{PVN2009}, as well as how to construct the symplectomorphisms $\varphi_{\alpha}, \alpha \in I$ such that $F_1 = F_2\circ \varphi_\alpha$ in cases 1), 2) and 3).
  For case 4) the symplectomorphism $\varphi_{\alpha}$ can be constructed as follows: instead of using~\cite[Theorem 2.1]{VN2003}, which gives a semi-local normal form for fibers containing exactly one focus-focus point, we use~\cite[Theorem 3.9]{PT}, which gives a semi-local normal form for fibers which contain any finite number of focus-focus points, with Equation~\eqref{eq:tilde-to-notilde}, which shows how to extract the invariant from~\cite{PT}  from the complete semitoric invariant.

  The proof in~\cite{PVN2009} concludes with \emph{Step 3} in which it is proven how to glue symplectically the semi-local symplectomorphisms in order to produce a global symplectomorphism $\varphi \colon M_1 \to M_2$.
  This step is unchanged in our case, since the existence of multiple focus-focus points in the same fiber does not play a role in the proof: only the local symplectomorphisms constructed in Step 2 are needed.
\end{proof}

\begin{remark}\label{rmk:map-injective}
 Proposition~\ref{prop:uniqueness} has two implications. The implication from left to right was discussed in Remark~\ref{rmk:quotientmap}.
 The other implication is equivalent to saying that the map $\mathrm{i}$ from Equation~\eqref{eqn:i} in Remark~\ref{rmk:quotientmap} is injective.
\end{remark}

\subsection{Existence} \label{sec:existence}

Recall that a vertex $v$ of a polygon in $\R^2$ is \emph{smooth} if the polygon is convex in a neighborhood of $v$ and inwards pointing normal vectors to two edges meeting at $v$ can be chosen so that they span the integral lattice $\Z^2$.
Also, given a polygon $\De$ we call the set 
$
 \{(x,y)\in\De\mid y\leq y'\textrm{ for all $y'$ such that }(x,y')\in\De\}
$
the \emph{lower boundary of $\De$} and we call the set 
$
 \{(x,y)\in\De\mid y\geq y'\textrm{ for all $y'$ such that }(x,y')\in\De\}
$
the \emph{upper boundary of $\Delta$}, see Figure~\ref{fig:upper-lower-boundary}.

\begin{figure}[ht]
\centering
\inputfigure{upper-lower-boundary}
  \caption{The upper and lower boundaries of a non-compact polygon $\Delta$.}
  \label{fig:upper-lower-boundary}
\end{figure}

\begin{definition} \label{def:ingred}
  Let $\mathbf{X}_{\lamf, \vf}$ be as in Equation~\eqref{eqn:defX}.
  A \emph{complete semitoric ingredient} is any element $\mathcal{I}$ 
  of the set $\coprod_{\lamf,\vf\geq 0} \left(\mathbf{X}_{\lamf,\vf}/(\Z^{\lambda_\mathrm{f}+1}\times\R)\right)$
  such that if $\mathcal{I}\in\mathbf{X}_{\lamf,\vf}/(\Z^{\lambda_\mathrm{f}+1}\times\R)$ is of the form
  \begin{equation*}
    \mathcal{I} = (\Z^{\lamf+1} \times \R) \cdot \left(\De, (\ell_{j_a})_{a = 1}^\lamf, (w_a)_{a = 1}^\lamf, (\tilde{c}_i)_{i = 1}^\vf, \left( m_i, [\tilde{\mathsf{s}}^i_\mu, \mathsf{g}^i_{\mu, \nu}]_{\mu, \nu \in \Z_{m_i}}\right)_{i = 1}^\vf\right)
  \end{equation*}
then the following properties are satisfied:
  \begin{enumerate}[itemsep = 0pt]
    \item \label{item:inequal} $\vf = \lamf = 0$ or $\vf\geq\lamf\geq 1$;
    \item \label{item:finiteheight} $\De \cap \ell_{j}$ is compact for all $j \in \R$, where $\ell_{j}$ is as in Definition~\ref{def:VPIA-group};
    \item \label{item:ck_lines} the entries of the $\vf$-tuple $(\tilde{c}_i)_{i=1}^{v_\mathrm{f}}$ are distinct, ordered
    lexicographically, and contained in $\mathrm{int}(\De) \cap \left( \bigcup_{a = 1}^\lamf \ell_{j_a}\right)$, and moreover $\ell_{j_a} \cap \Seq{\tilde{c}_i}_{i = 1}^\vf\neq\varnothing$ for all $a \in \Seq{1, \dotsc, \lamf}$; 
    \item \label{item:smooth} every vertex of $\De$ in $\De \setminus \left(\cup_{a=1}^{\lambda_\mathrm{f}}\ell_{j_a}\right)$ is smooth;
    \item \label{item:hidden} for every $a \in \Seq{1, \dotsc, \lamf}$ if $P \in \partial\De \cap \ell_{j_a}$ then:
    \begin{enumerate}[nosep]
      \item if $P$ is in the lower boundary of $\De$ then $\mathsf{t}_{j_a}^{-w_a}(\De)$ either has no vertex at $Q=\mathsf{t}_{j_a}^{-w_a}(P)$ or the vertex at $Q$ is smooth;
      \item if $P$ is in the upper boundary of $\De$ then $\mathsf{t}_{j_a}^{-(w_a+\tilde{m}_{j_a})}(\De)$ either has no vertex at $Q=\mathsf{t}_{j_a}^{-(w_a+\tilde{m}_{j_a})}(P)$ or the vertex at $Q$ is smooth, where $\tilde{m}_{j_a} = \sum_{i, c_i \in \ell_{j_a}} m_i$;
    \end{enumerate}
    \item \label{item:constterm} if for every $i\in\{1,\ldots, \vf\}$ we let $\tilde{c}_i = (\tilde{c}_i^1,\tilde{c}_i^2)$, then $(\tilde{\mathsf{s}}^i_\mu)^{(0,0)} = 2\pi \tilde{c}_i^2$ for all $\mu \in \Z_{m_i}$;
    \item \label{item:taylorcond} for every $i\in\{1,\ldots,\vf\}$ the tuple $[\tilde{\mathsf{s}}^i_\mu, \mathsf{g}^i_{\mu, \nu}]_{\mu,\nu\in\Z_{m_i}}$ satisfies the conditions in Equation~\eqref{eq:invariant-constraint}.
  \end{enumerate}
\end{definition}

In view of Definition~\ref{def:ingred} we denote by $\mathbf{X}$ the set
of complete semitoric ingredients, which is a proper subset
of $\coprod_{\lambda_\mathrm{f},v_\mathrm{f}\in\Z_{\geq 0}} \left( \mathbf{X}_{\lamf, \vf} / (\Z^{\lambda_\mathrm{f}+1}\times\R)\right)$.

\begin{remark}
  Definition~\ref{def:ingred} generalizes to non-simple semitoric systems the Pelayo--V\~{u} Ng\d{o}c semitoric list of ingredients given by~\cite[Definition 4.5]{PVN2011}.
\end{remark}

An example of a complete semitoric ingredient appears in Figure~\ref{fig:complete-invariant}.

\begin{proposition}
  Let $(M, \om, F)$ be a semitoric system.
  Then $(\Z^{\lamf+1} \times \R) \cdot \tilde{\mathrm{i}}(M, \om, F)$, where $\tilde{\mathrm{i}}(M, \om, F)$ is given in Equation~\eqref{eq:object}, satisfies conditions (1)-(7) in Definition~\ref{def:ingred}.
\end{proposition}

\begin{proof}
  Item~\eqref{item:inequal} holds since the cardinality of $F(M_{\mathrm{f}})$ is $\vf$ and the cardinality of $\proj_1(F({M_\mathrm{f}}))$ is $\lamf$.
  Let $\mathcal{A} = (\mathcal{A}^1, \mathcal{A}^2)$ be the choice of piecewise affine coordinates as in
  Definition~\ref{def:affinecoords} such that $\De = \mathcal{A}(B)$.

  Item~\eqref{item:finiteheight} holds because $J=\mathcal{A}^1\circ F$ is proper.

  Item~\eqref{item:ck_lines} holds because the $\tilde{c}_i$ are obtained as the images under $\mathcal{A}$ of the focus-focus values, which lie in $\mathrm{int}(\De)$, and $j_1, \dotsc, j_{\lamf}$ are defined as the elements of $\proj_1(\{ c_1,\ldots,c_{v_\mathrm{f}}\}) = \proj_1(\{ \tilde{c}_1,\ldots,\tilde{c}_{v_\mathrm{f}}\})$.

  Item~\eqref{item:smooth} is immediate since $\mathcal{A}\circ F \colon M \to \R^2$ is a toric momentum map away from $F^{-1}(\bigcup_{a=1}^\lamf \ell_{j_a})$.

  Similarly, Item~\eqref{item:hidden} follows from the fact that if the wall-crossing index is zero for some segment of $\ell_{j_a}$ then the piecewise affine coordinates are smoothly continued across that region of the wall, so again $\mathcal{A}\circ F$ is  locally a toric momentum map.
  The polygons $\mathsf{t}_{j_a}^{-w_a}(\De)$ and $\mathsf{t}_{j_a}^{-(w_a+\tilde{m}_{j_a})}(\De)$ considered in the two parts of Item~\eqref{item:hidden} are formed by choosing different piecewise affine coordinates for which the wall-crossing index near $\mathsf{t}_{j_a}^{-w_a}(P)$,
  respectively $\mathsf{t}_{j_a}^{-(w_a+\tilde{m}_{j_a})}(P)$, is zero.

  Item~\eqref{item:constterm} holds because 
  $
   (\tilde{\mathsf{s}}^i_\mu)^{(0,0)} = \tilde{S}^i(c_i) = 2\pi \mathcal{A}^2(c_i) = 2\pi \tilde{c}_i^2
  $
  by Equation~\eqref{eq:def-S} and the fact that $\tilde{c}_i = \mathcal{A}(c_i)$.

  Item~\eqref{item:taylorcond} follows from Lemma~\ref{lem:invariant-constraint}.
\end{proof}

The following extends~\cite[Theorem 4.6]{PVN2011} to the non-simple case.

\begin{proposition} \label{prop:existence}
  Given a complete semitoric ingredient $\mathcal{I}$, as in Definition~\ref{def:ingred}, there exists a semitoric system $(M, \om, F)$ such that the complete semitoric invariant of $(M, \om, F)$ is $\mathcal{I}$.
\end{proposition}

\begin{proof}
  Given a complete semitoric ingredient $\mathcal{I}$ choose a representative such that $w_a = 0$ for $a\in\{1,\ldots, \lamf\}$ so that
  \begin{equation*}
    \mathcal{I} = (\Z^{\lamf+1} \times \R) \cdot \left( \De, (\ell_{j_a})_{a = 1}^\lamf, (0)_{a = 1}^\lamf, (\tilde{c}_i)_{i = 1}^\vf, \left( m_i, [\tilde{\mathsf{s}}^i_\mu, \mathsf{g}^i_{\mu, \nu}]_{\mu, \nu \in \Z_{m_i}}\right)_{i = 1}^\vf\right).
  \end{equation*}
  Note that such a choice of representative always exists because the action of $\Z^{\lambda_\mathrm{f}+1}\times\R$ can be used to make the tuple of
  lower wall-crossing indices take any desired value, as seen in Equation~\eqref{eq:Gj-action}.
  Now we continue as in the proof of~\cite[Theorem 4.6]{PVN2011}, which proceeds by gluing together the semi-local models of the fibers of $F$, 
  essentially constructing $(M, \om, F)$ backwards starting from $\De$
  and using the semi-local models and symplectic gluing to construct
  $(M,\om)$ and a map $\mu\colon M\to\De$ which will represent $F\circ \mathcal{A}$
  for some choice of piecewise affine coordinates $\mathcal{A}$.
  The proof of~\cite[Theorem 4.6]{PVN2011} is split into four stages and we will consider each separately.

  In the preliminary stage (\emph{a convenient covering}) and first stage (\emph{away from the cuts}) of the proof of~\cite[Theorem 4.6]{PVN2011} one constructs a convenient covering $\Seq{\Omega_{\alpha}}_{\alpha \in I}$ of $\Delta$.
  In~\cite{PVN2011}, following~\cite{VN2007}, the polygon invariant is constructed by choosing rays in the momentum map image known as \emph{cuts} which
  go either up or down from each focus-focus point, and then finding a toric momentum map on the manifold with preimages under $F$ of these cuts removed.

  In the first stage one restricts to the subcovering $\Seq{\Omega_{\alpha}}_{\alpha \in I'}$ of sets which do not intersect the cuts, and for each of these constructs a local symplectic model $M_{\alpha}$ and an integrable system 
  \[
  F_{\alpha} \colon M_{\alpha} \to \Omega_{\alpha}.
  \]
  In the language of the present paper, we replace the cuts referred to above by $\bm{\ell}^{\mathrm{nonzero}}$, where $\bm{\ell}^{\mathrm{nonzero}}$ is the union of the portions of the lines $\ell_{j_a}$, $a\in\{1,\ldots,\lambda_\mathrm{f}\}$, which have non-zero wall-crossing index, as in Definition~\ref{def:wallcrossing}.
  Since we have chosen a representative for which the all of the lower wall-crossing indices are zero, the set $\De \setminus \bm{\ell}^{\mathrm{nonzero}}$ is connected.
  After making this choice, the remainder of the first stage continues exactly as in~\cite[pages 113--116]{PVN2011}.
  Using the general symplectic gluing theorem~\cite[Theorem 3.11]{PVN2011} these local regular integrable models can be be symplectically glued together in order to produce an integrable system 
  \[
  F_{I'} \colon M_{I'} \to \bigcup_{\alpha \in I'} \Omega_{\alpha}
  \]
  over the union of the subcovering $\Seq{\Omega_{\alpha}}_{\alpha \in I'}$ of open sets which do not intersect
  $\bm{\ell}^{\mathrm{nonzero}}$.

  In the second stage (\emph{attaching focus-focus fibrations}) of~\cite[pages 116--118]{PVN2011} it is explained how to symplectically glue the semi-local models in a neighborhood of the focus-focus fibers containing exactly one focus-focus point to the model  $F_{I'} \colon M_{I'} \to \bigcup_{\alpha \in I'} \Omega_{\alpha}$ of Step 1, to produce a proper map $F_{I''} \colon M_{I''} \to \bigcup_{\alpha \in I''}\Omega_{\alpha}$ on the symplectic manifold $M_{I''}$, which is a smooth toric momentum map away from the pre-images of the cuts.
  Here $\Seq{\Omega_{\alpha}}_{\alpha \in I''}$ is the subcovering containing all of $\Seq{\Omega_{\alpha}}_{\alpha \in I'}$
  and additionally containing those elements of $\Seq{\Omega_{\alpha}}_{\alpha \in I}$ which contain exactly one focus-focus value.

  This same construction can be done for non-simple semitoric systems taking into account the following: fix some $i\in\{1,\ldots, \vf\}$ and consider the marked point $\tilde{c}_i = (\tilde{c}_i^1,\tilde{c}_i^2) \in \De$.
  Using~\cite[Theorem 3.9]{PT} construct the semi-local model $(M_i, \omega_i, F_i)$ over a neighborhood of the origin in $\R^2$ using the invariant $(\mathsf{s}^i_\mu, \mathsf{g}^i_{\mu, \nu})_{\mu, \nu \in \Z_{m_i}}$ obtained from $\mathcal{I}$ as in Equation~\eqref{eq:tilde-to-notilde}, well-defined up to cyclic reordering of the indices.
  Let $B_i= F_i(M_i)$ be the base and let $(p^i_\mu)_{\mu \in \Z_{m_i}}$ be the tuple of focus-focus points of $F_i$, where we assume that $F_i(p^i_\mu) = 0$ for all $\mu \in \Z_{m_i}$.
  For each $\mu \in \Z_{m_i}$ there is a symplectomorphism 
  \[
   \varphi^i_\mu \colon (M_i, \omega_i, p^i_\mu) \to (\R^4, \omega_0, 0)
  \] and a diffeomorphism $E^i_\mu \colon (B_i, 0) \to (\R^2, 0)$ such that $q \circ \varphi^i_\mu = E^i_\mu \circ F_i$.
  Let $\tilde{S}^i \colon B_i \to \R$ be a smooth function such that Equation~\eqref{eq:deftildes} holds for one choice of $\mu$, and due to the relations~\eqref{eq:invariant-constraint} it thus holds for all choices of $\mu$.
  In order to obtain a system with the desired Taylor series invariants define $\mathcal{A}_i = (\mathcal{A}_i^1,\mathcal{A}_i^2) \colon B_i \to \R^2$ by
  \begin{align} \label{eqn:existence_affine}
  \left\{ \begin{array}{l}
    \mathcal{A}_i^1(x,y)  =   x + \tilde{c}_i^1,  \\
    \mathcal{A}_i^2(x,y)  =   \frac{1}{2\pi}\tilde{S}^i + \frac{1}{2\pi} \sum_{\nu \in \Z_{m_i}} (E^i_\nu)^* K_+ - w_a^{\alpha} (\tilde{c}^2_i-x) \cdot \Heaviside_{(\tilde{c}^2_i-x)},
   \end{array}\right.
  \end{align}
  where $\tilde{c}_i$ is the $\alpha^{\mathrm{th}}$ marked point on the line $\ell_{\tilde{c}^1_i}$, counting up from the bottom.
  The map $\mathcal{A}_i$ is invariant under cyclic reordering of the indices in the Taylor series.

  Now we use $\mathcal{A}_i \circ F_i \colon M_i \to \De$ to place $B_i$ into a neighborhood of $\tilde{c}_i$ in the polygon $\Delta$, and we perform gluing as in~\cite[pages 116--118]{PVN2011}.
  Here note that $\mathcal{A}_i\circ F_i$ replaces the map $R_\alpha\circ g_i\circ F_i$ from~\cite[page 117]{PVN2011}, in which $g_i$ is a smooth diffeomorphism of $\R^2$ analogous to $\mathcal{A}_i$ and the map $R_\alpha$ was used to account for the twisting index, which in the present proof is already accounted for in the piecewise affine coordinates $\mathcal{A}_i$ since the information of the twisting index is included in the new Taylor series $[\tilde{\mathsf{s}}^i_\mu, \mathsf{g}^i_{\mu, \nu}]_{\mu, \nu \in \Z_{m_i}}$.
  From Equation~\eqref{eqn:existence_affine} it follows that the the Taylor series obtained from the constructed system will be the desired one, since
  isolating $\tilde{S}^i$ in Equation~\eqref{eqn:existence_affine} yields the definition of the desired action Taylor series as in Equation~\eqref{eq:def-S}.

  In the third (\emph{filling in the gaps}) of the proof one considers the remaining open sets $\Omega_{\alpha}$ in the covering, which are those including the cuts but not any focus-focus values, and includes them into the previous gluing data using symplectic gluing in order to obtain a symplectic manifold and a proper map $\mu \colon M \to \bigcup_{\alpha \in I}\Omega_{\alpha}$ with image $\Delta$.
  This map $\mu$ is a proper toric smooth momentum map only away from the cuts, and in the 
  fourth and final stage (\emph{recovering smoothness}) the authors show how to smoothen $\mu$.
  In the case of non-simple semitoric systems these final two stages proceed exactly as in~\cite{PVN2011}, using different choices of representative for $\mathcal{I}$ in order to make the wall-crossing index of the vertical lines $\ell_{j_a}$, $a\in\{1,\ldots,\lambda_\mathrm{f}\}$, equal to zero around the remaining points to be glued in, which are the points on the lines $\ell_{j_a}$ which do not already have zero wall-crossing index in the representative of $\mathcal{I}$ we started with.
\end{proof}

\begin{remark}\label{rmk:convexity}
  Let $\mathcal{I}$ be as in Definition~\ref{def:ingred}.
  Not every polygon $\De$ such that
  \begin{equation*}
    \mathcal{I} = (\Z^{\lamf+1} \times \R) \cdot \left( \De, (\ell_{j_a})_{a = 1}^\lamf, (w_a)_{a = 1}^\lamf, (\tilde{c}_i)_{i = 1}^\vf, \left( m_i, [\tilde{\mathsf{s}}^i_\mu, \mathsf{g}^i_{\mu, \nu}]_{\mu, \nu \in \Z_{m_i}}\right)_{i = 1}^\vf\right)
  \end{equation*}
  is convex, but the conditions on the vertices (Items~\eqref{item:smooth} and~\eqref{item:hidden}) imply that the polygon associated to a representative
  is convex if $w_a^0 \leq 0$ and $w_a^{s_a} \geq 0$ for $1 \leq a \leq \lamf$, where $s_a$ is the number of focus-focus values in line $\ell_{j_a}$.
  For instance, the polygon associated to any representative for which $w_0 = \dotsb = w_a = 0$ is convex.
  One could restrict to only the convex representatives without losing any information, but there is no 
  natural reason to exclude the other polygons so we keep all representatives as the invariant.
  In Section~\ref{sec:equiv-classification} we discuss an alternative (but equivalent)
  way to encode the complete semitoric invariant in which all polygons are convex.
\end{remark}

\begin{remark}
  In~\cite{PVN2011} the authors describe \emph{hidden} and \emph{fake corners} of the polygon, which represent the two possible cases in Item~\eqref{item:hidden} above.
  A vertex which occurs on a line $\ell_{j_a}$ is a fake corner if there is no vertex there after changing the piecewise affine coordinates so that the adjacent wall-crossing index is zero, and such a vertex is a hidden corner if there is a smooth vertex remaining after changing to the appropriate coordinates.
  In Figure~\ref{fig:complete-invariant} the bottom right vertex on the line $\ell_{j_3}$ is hidden, since the slope of the bottom boundary
  changes by $5$ even though the adjacent wall-crossing index is only $-4$, and the rest of the vertices on the lines $\ell_{j_1}$, $\ell_{j_2}$, and $\ell_{j_3}$
  are fake corners, since the changes in slope correspond to the adjacent wall-crossing indices.
\end{remark}

\begin{remark}\label{rmk:map-surjective}
Proposition~\ref{prop:existence} says that the injective
 map $\mathrm{i}\colon \mathcal{M}/{\sim}\to\mathbf{X}$ from Equation~\eqref{eqn:i} discussed in Remarks~\ref{rmk:quotientmap} and~\ref{rmk:map-injective} is also surjective.
\end{remark}

\subsection{Classification}
 \label{sec:mainresult}

The following classification generalizes the Pelayo--V\~{u} Ng\d{o}c classification of simple semitoric systems~\cite{PVN2009, PVN2011} by allowing the fibers of $J$ (and hence of $F$) to have multiple focus-focus points per fiber.
This includes fibers such as $F^{-1}(c_2)$ and $F^{-1}(c_3)$ in Figure~\ref{fig:fibers}. The proof follows from Propositions~\ref{prop:welldef}, \ref{prop:uniqueness}, and~\ref{prop:existence}. 

\begin{theorem} \label{thm:classification}
  For each complete semitoric ingredient as in Definition~\ref{def:ingred} there exists a semitoric integrable system with that as its complete semitoric invariant as in Definition~\ref{def:complete}.
  Moreover, two semitoric systems are isomorphic if and only if they have the same complete semitoric invariant.
\end{theorem}

\begin{figure}[ht]
\centering
\inputfigure{fibers}
  \caption{Focus-focus fibers of a semitoric system $F\colon M\to \R^2$ where
  $M$ is a symplectic $4$-manifold. The piecewise affine coordinates $\mathcal{A}$ (as in Definition~\ref{def:affinecoords}) map the momentum map image
  onto a polygon. The system shown has $v_\mathrm{f} = 3$, $\lambda_\mathrm{f}=3$, $m_1 = 1$, $m_2=5$, and $m_3=7$.}
  \label{fig:fibers}
\end{figure}

\begin{remark}
The proof method of Theorem~\ref{thm:classification} follows closely, as we have seen, the articles~\cite{PVN2009, PVN2011}. The new contribution consists of making sure that the proof strategy can be implemented with the more general semitoric invariant, which as we have seen involves certain difficulties, both conceptual and technical (due precisely to the presence of fibers with multiple pinched points).

In the present paper 
we described how the symplectic
invariants in Defintion~\ref{def:Taylor_series} (the Taylor series invariants) constructed in~\cite{PT} relate to the original construction of the twisting index invariant. 
In Section~\ref{sec:taylorseries}, we showed how
they can be naturally packaged together into a single invariant which mixes the information of both original invariants  
(in~\cite{Jaume-thesis} the invariants of simple semitoric systems were combined in a similar way).

In Section~\ref{ssec:semitoricpolygon}, we described how the multipinched fibers
change the induced affine structure, and that our analogue of the ``polygonal invariant" of the Pelayo--V\~{u} Ng\d{o}c classification (the third invariant) may no longer be convex (Remark~\ref{rmk:convexity}), as illustrated in Figure~\ref{fig:complete-invariant}. 
Note that there is no deep reason for the non-convexity of this polygon. We have allowed for any valid integral affine coordiantes
in each vertical strip, but by requiring certain compatibility of the coordinates in each strip we could obtain a set of convex polygons, just as in
the case of simple semitoric systems.
This is the case in Section~\ref{sec:equiv-classification}, in which we use Theorem~\ref{thm:classification} to prove Theorem~\ref{thm:equiv-classification}, which states that semitoric systems may equivalently be classified by an invariant which only includes convex polygons.
The polygonal invariant was the complete invariant of the classification of compact toric systems due to Atiyah--Guillemin--Sternberg--Delzant~\cite{At82,GS82,De88}, because compact toric systems cannot have focus-focus points;  see~\cite{Pesurvey, GuSj2005} for an expository account. 
For the case of non-compact toric systems see~\cite{KaLe2015}.
  It would be interesting to extend the classification to systems having ``hyperbolic triangles" as in  \cite{DP2016, HP-extend} and \cite[Section 6.6]{LFPfamilies}.
\end{remark}

\begin{remark}
 \label{rmk:recover}
In the case that $(M,\om,F)$ is a simple semitoric system the five original invariants (1)--(5) 
from~\cite{PVN2009,PVN2011} can be obtained from the complete semitoric invariant given in Definition~\ref{def:complete}.
(1) The \emph{number of focus-focus points} is equal to $\lamf = \vf = m_{\mathrm{f}}$. 
(2) From the Taylor series labels on each
focus-focus value in the complete semitoric invariant one can extract the Taylor series invariant 
from~\cite{PT} via Equation~\eqref{eq:tilde-to-notilde}, which in the simple case is determined
by a single series $\tilde{\mathsf{s}}_{0}^i$ for each focus-focus value $c_i$.
The relationship between this series and the \emph{Taylor series invariant} $(S_i)^\infty$
from~\cite{PVN2009} is 
\[
 \widetilde{\mathsf{s}}_{0}^i (X,Y) = (S_i)^\infty (Y,X) + \frac{\pi}{2} X \pmod{2\pi X}
\]
where the addition of $(\pi / 2) X$ is due to a change in convention
between~\cite{PT} and~\cite{VN2003}, as discussed in~\cite[Remark 3.10]{PT}. 
(3) The \emph{semitoric polygon invariant} in~\cite{PVN2009}
is obtained by taking the image of a toric momentum map defined on the complement in $M$ of the preimages
under $F$ of rays which start at each focus-focus value and go either up or down, these are known as ``cuts'' in the base space.
These polygons correspond to the subset of images of piecewise affine coordinates $\mathcal{A} = (\mathcal{A}^1,\mathcal{A}^2)$ (as in Definition~\ref{def:affinecoords}) such that
for each $a\in \{1,\ldots, \lambda_\mathrm{f}\}$ the lower wall-crossing index $w_a$ satisfies either 
$w_a=0$ (corresponding to an upwards cut) or $w_a = -1$ (corresponding to a downwards cut).
(4) The \emph{height invariant} $\mathsf{h}^i$ of the focus-focus value $c_i = (c_i^1,c_i^2)$ is the distance from
the marked point $\tilde{c}_i$ to the bottom of the corresponding polygon, obtained by
 \[\mathsf{h}^i = \frac{1}{2\pi}\left(\tilde{\mathsf{s}}_{0}^i \right)^{(0, 0)} - \min\limits_{\ell_{c_i^1}\cap \De} \mathcal{A}^2.\]
By Lemma~\ref{lem:Gaction-Taylor}, $\mathsf{h}^i$ does not depend on the choice of $\mathcal{A}$.
(5) Finally, the \emph{twisting index invariant} $\mathsf{k}^i_{\mathrm{classical}}$ was originally
defined in~\cite{PVN2009} by comparing $\mathcal{A}\circ F$ with a local preferred momentum map, and is essentially the integer part
of $\frac{1}{2\pi}(\tilde{\mathsf{s}}_{0})^{(1,0)}$, but again there is a shift by
$(\pi/2)X$, so the twisting index invariant of $c_i$ is given as
\[
 \mathsf{k}^i_{\mathrm{classical}} = \left\lfloor \frac{1}{2\pi} \left( (\tilde{\mathsf{s}}_{0}^i)^{(1,0)} - \frac{\pi}{2} \right) \right\rfloor + \frac{\epsilon_i-1}{2}
\]
where $\lfloor\cdot\rfloor\colon\R\to\Z$ is the usual floor function
and $\epsilon_i = +1$ if the cut at $c_i$ is upward and $\epsilon_i = -1$ if the
cut at $c_i$ is downward.
Note that this integer label on each $c_i$ does depend on the choice of piecewise affine coordinates, since
changing piecewise affine coordinates can shift the coefficient of $X$ in $\tilde{\mathsf{s}}_{0}^i$
by an integer multiple of $2\pi$, as is seen in Lemma~\ref{lem:Gaction-Taylor}.
The last term of $\mathsf{k}^i_{\mathrm{classical}}$ is there so that it is preserved under a change
in cut direction at $c_i$,
in~\cite{PVN2009} the dependence of the preferred momentum map on the cut direction was
designed so that this would hold.
For further discussion of the twisting index see Section~\ref{sec:twist}.
\end{remark}

\section{An equivalent classification}
\label{sec:equiv-classification}

The complete semitoric invariant can also be encoded in a way which is more consistent with the original classification~\cite{PVN2009, PVN2011}, which we present here.
We will now define a new object, a \emph{marked labeled semitoric polygon}, and then define a bijective map from the set of these objects to the set of complete semitoric invariants. From there, we apply Theorem~\ref{thm:classification} to obtain a bijection between the set of marked labeled semitoric polygons and the set of isomorphism classes of semitoric systems, simple or not.

Let $v$ be a vertex of a convex polygon, and let $u_1,u_2\in\Z^2$ be the primitive vectors directing the edges emanating from $v$. Let $s\in\Z_{>0}$. Then:
\begin{itemize}[noitemsep]
    \item we say that $v$ satisfies the \emph{$s$-fake} condition if $\det(u_1,T^su_2)=0$; and
    \item we say that $v$ satisfies the \emph{$s$-hidden} condition if $\det(u_1,T^su_2)=\pm 1$.
\end{itemize}
For $\epsilon\in\{-1,1\}$ and $c = (c^1,c^2)\in\R^2$, define
\[
 \ell_c^\epsilon = \{(x,y)\in\R^2 \mid x=c_1 \text{ and }\epsilon y 
 \geq \epsilon c_2\},
\]
so $\ell_c^\epsilon$ is the ray starting at $c$ which goes up if $\epsilon=1$ and down if $\epsilon = -1$.
We call these rays \emph{cuts}.


\begin{definition}\label{def:markedpoly-rep}
A \emph{marked labeled semitoric polygon representative} is a tuple \[\left( \De, (c_{n})_{{n}=1}^\mf, (\epsilon_n)_{n=1}^\mf, \left( m_i, [\tilde{\mathsf{s}}^i_\mu, \mathsf{g}^i_{\mu, \nu}]_{\mu, \nu \in \Z_{m_i}}\right)_{i = 1}^\vf\right)\] where
$\De$ is a convex polygon such that
\begin{enumerate}[noitemsep]
    \item \label{item:markedpoly1} $\mf\geq 0$, $\vf = \#\{c_n\}_{n=1}^\mf$,  $\{\tilde{c}_i\}_{i=1}^\vf = \{c_n\}_{n=1}^\mf$, the $\tilde{c}_i$ are distinct and in lexicographic order, and \[m_i = \#\{n\in\{1,\ldots,\mf\} \mid c_n = \tilde{c}_i\};\] 
    \item  $\De \cap \ell_{j}$ is compact for all $j \in \R$, where $\ell_{j}$ is as in Definition~\ref{def:VPIA-group};
    \item $c_1,\ldots,c_\mf\in\mathrm{int}(\De)$ and $c_1,\ldots,c_\mf$ are in lexicographic order;
    \item each point $v$ of $\partial \De \cap \left(\cup_i \ell_{c_i}^{\epsilon_i}\right)$ is a vertex which satisfies either the $s$-hidden or $s$-fake corner condition, where
    \[
        s = \#\{i\in\{1,\ldots, \mf\}\mid v\in \ell_{c_i}^{\epsilon_i}\}
    \]
    and all other vertices of $\De$ are smooth;
    \item  if for every $n\in\{1,\ldots, \mf\}$ we let $c_n = (c_n^1,c_n^2)$, then $(\tilde{\mathsf{s}}^i_\mu)^{(0,0)} = 2\pi c_n^2$ for all $\mu \in \Z_{m_i}$ whenever $c_n = \tilde{c}_i$;
    \item \label{item:markedpoly-last}  for every $i\in\{1,\ldots,\vf\}$ the tuple $[\tilde{\mathsf{s}}^i_\mu, \mathsf{g}^i_{\mu, \nu}]_{\mu,\nu\in\Z_{m_i}}$ satisfies the conditions in Equation~\eqref{eq:invariant-constraint}.
\end{enumerate}
\end{definition}

Notice that in Definition~\ref{def:markedpoly-rep} we have not assumed that the $c_i$ are distinct.
Furthermore, notice that the polygon $\De$ above is now required to be \emph{convex}.
A marked labeled semitoric polygon representative can be thought of as a single convex polygon with a finite number of marked points, some of which may be equal, and cuts going up or down out of each marked point, with each \emph{distinct} marked point labeled by Taylor series information (which now includes the data of the twisting index invariant).

Let $G^\mf = \{1,-1\}^\mf$ and let $\mathcal{T}$ denote the set of integral affine transformations of $\R^2$ which preserve the first component, so $\mathcal{T} = \{\shift_b\circ T^{z_0} \mid b\in\R, z_0\in\Z\}$.
Then $G^\mf\times\mathcal{T}$ acts on a marked labeled semitoric polygon representative by 
\begin{align}\label{eqn:markedpoly-groupaction} ((\epsilon_n')_{n=1}^\mf,\shift_b\circ T^{z_0}) \cdot &\left( \De, (c_n)_{n=1}^\mf, (\epsilon_n)_{n=1}^\mf, \left( m_i, [\tilde{\mathsf{s}}^i_\mu, \mathsf{g}^i_{\mu, \nu}]_{\mu, \nu \in \Z_{m_i}}\right)_{i = 1}^\vf\right) =\\ \nonumber
&\left( \tau(\De), (\tau(c_n))_{n=1}^\mf, (\epsilon_n'\epsilon_n)_{n=1}^\mf, \left( m_i, [\tau(\tilde{\mathsf{s}}^i_\mu), \mathsf{g}^i_{\mu, \nu}]_{\mu, \nu \in \Z_{m_i}}\right)_{i = 1}^\vf\right)
\end{align}
where \[\tau = \shift_b\circ \mathsf{t}_{\pi_1(c_1)}^{u_1}\circ\ldots\circ \mathsf{t}_{\pi_1(c_\mf)}^{u_\mf}\circ T^{z_0}\]
and $u_n = (\epsilon_n+\epsilon_n \epsilon_n')/2$ for $n\in\{1,\ldots,\mf\}$.

\begin{definition}\label{def:markedpoly}
A \emph{marked labeled semitoric polygon} is the orbit of a marked labeled semitoric polygon representative (as in Definition~\ref{def:markedpoly-rep}) under the group action given in Equation~\eqref{eqn:markedpoly-groupaction}.
Let $\mathbf{Y}$ denote the set of marked labeled semitoric polygons.
\end{definition}
It can be checked that every representative of a marked labeled semitoric polygon satisfies the conditions \eqref{item:markedpoly1}--\eqref{item:markedpoly-last} above to be a marked labeled semitoric polygon representative.

Now we can define the bijection between marked labeled semitoric polygons and the complete semitoric invariant.
Recall that $\mathbf{X}$ denotes the set of complete semitoric ingredients.
Let $\Phi\colon \mathbf{Y}\to\mathbf{X}$ be the map which sends 
the marked labeled semitoric polygon
\begin{equation}\label{eqn:markedpolymap-1}
(G^s\times\mathcal{T})\cdot \left( \De, (c_n)_{n=1}^\mf, (\epsilon_n)_{n=1}^\mf, \left( m_i, [\tilde{\mathsf{s}}^i_\mu, \mathsf{g}^i_{\mu, \nu}]_{\mu, \nu \in \Z_{m_i}}\right)_{i = 1}^\vf\right)
\end{equation}
to the complete semitoric invariant
\begin{equation}\label{eqn:markedpolymap-2}
(\Z^{\lamf+1} \times \R) \cdot \left(\De, (\ell_{j_a})_{a = 1}^\lamf, (w_a)_{a = 1}^\lamf, (\tilde{c}_i)_{i = 1}^\vf, \left( m_i, [\tilde{\mathsf{s}}^i_\mu, \mathsf{g}^i_{\mu, \nu}]_{\mu, \nu \in \Z_{m_i}}\right)_{i = 1}^\vf\right)
\end{equation}
where $\{j_1,\ldots,j_{\lamf}\} = \{\pi_1(c_1),\ldots,\pi_1(c_\mf)\}$ with $j_1<\ldots < j_{\lamf}$, and
\[
 w_a = \#\{n \in \{1,\ldots, \mf\} \mid \epsilon_n = -1 \text{ and } c_n \in \ell_{j_a}\}
\]
for $a\in\{1,\ldots, \lamf\}$.
That is, $w_a$ counts the number of downwards cuts in the line $\ell_{j_a}$.
It is straightforward to see that the resulting object satisfies the conditions given in Definition~\ref{def:ingred} to be a complete semitoric ingredient and that it is independent of the choice of representative (and thus well-defined).

Using Theorem~\ref{thm:classification}, we are now ready to prove:

\begin{theorem}\label{thm:equiv-classification} 
Let $\mathrm{i}$ be the map from Equation~\eqref{eqn:i} which sends an isomorphism class of semitoric systems to the associated complete semitoric invariant, and let $\Phi$ be the map described in Equations~\eqref{eqn:markedpolymap-1} and~\eqref{eqn:markedpolymap-2} which sends a complete semitoric invariant to the associated marked labeled semitoric polygon.
 Then \[\Phi^{-1}\circ \mathrm{i}\colon \mathcal{M}/{\sim} \,\to \mathbf{Y}\] is a bijection between the set of isomorphism classes of semitoric systems and the set of marked labeled semitoric polygons.
\end{theorem}

\begin{proof}
By Theorem~\ref{thm:classification} we know that $\mathrm{i}$ is a bijection, and thus to prove the theorem it is sufficient
to prove that $\Phi$ is a bijection. It is clear that $\Phi$ is surjective.

To prove that $\Phi$ is injective suppose that it sends two elements of $\mathbf{X}$ to the same element of $\mathbf{Y}$, and we will show that the two elements of $\mathbf{X}$ are equivalent up to the action of $G^s\times \mathcal{T}$.
It is immediate that, up to the action of $G^s\times\mathcal{T}$, the two elements of $\mathcal{M}/\sim$ have the same polygon,
the same marked points, and the same Taylor series information.
Thus we may write them as
\[
(G^s\times\mathcal{T})\cdot \left( \De, (c_n)_{n=1}^\mf, (\epsilon_n^{p})_{n=1}^\mf, \left( m_i, [\tilde{\mathsf{s}}^i_\mu, \mathsf{g}^i_{\mu, \nu}]_{\mu, \nu \in \Z_{m_i}}\right)_{i = 1}^\vf\right)
\]
for $p\in \{1,2\}$ where 
\begin{equation}\label{eqn:wall-crossing-equivalent}
 \#\{n \in \{1,\ldots, \mf\} \mid \epsilon_n^1 = -1, c_n \in \ell_{j_a}\} = \#\{n \in \{1,\ldots, \mf\} \mid \epsilon_n^2 = -1, c_n \in \ell_{j_a}\}
\end{equation}
for each $a \in \{1,\ldots, \lamf\}$. It may not be true that $\epsilon_n^1 = \epsilon_n^2$ for all $n\in\{1,\ldots,\mf\}$, but we will now show that we can arrange this by passing to an equivalent representative in a way that doesn't change any of the other components of the invariant.

Let $a \in \{1,\ldots, \lamf\}$ and let $c_{n_1}, c_{n_2}\in \ell_a$.
Suppose that $\epsilon_{n_1}^2 = - \epsilon_{n_2}^2$, so the corresponding cuts are in opposite directions.
 Define
$
 \epsilon' = (\epsilon_1',\ldots, \epsilon_s') \in G^s 
$
by
\[
 \epsilon_n' = \begin{cases} -1, & \text{ if }n = n_1 \text{ or }n=n_2 \\ 1, & \text{ otherwise},\end{cases}
\]
so the action of $\epsilon'$ switches the directions of the $n_1^{\text{th}}$ and $n_2^{\text{th}}$
cuts and leaves all other cuts unchanged.
Then, following Equation~\eqref{eqn:markedpoly-groupaction}, we see that $\epsilon'$ acts on $\De$, $(c_n)_{n=1}^{\mf}$, and the Taylor series labels by $\tau = \mathsf{t}_{j_a}^{1}\circ\mathsf{t}_{j_a}^{-1}$ which is the identity map.
We conclude that interchanging the cut directions of two marked points which lie on the same vertical line and leaving all other components unchanged yields a marked labeled semitoric polygon which is equivalent under the action of $G^s\times\mathcal{T}$.

For each $a \in \{1,\ldots, \lamf\}$ by Equation~\eqref{eqn:wall-crossing-equivalent} we see that the total number of downwards cuts in the line $\ell_a$ is equal in each element, so there exists a permutation $\sigma$ of $\{n\in\{1,\ldots, \mf\} \mid c_n \in \ell_{j_a}\}$ such that $\epsilon_n^1 = \epsilon_{\sigma(n)}^2$ for all such $n$. 
The permutation $\sigma$ can be written as a sequence of transpositions.
Due to the conclusion of the previous paragraph, and using $\sim$ to denote equivalence up to the action of $G^s\times\mathcal{T}$, this implies that  
\begin{align*}
\left( \De, (c_n)_{n=1}^\mf, (\epsilon_n^{1})_{n=1}^\mf, \left( m_i, [\tilde{\mathsf{s}}^i_\mu, \mathsf{g}^i_{\mu, \nu}]_{\mu, \nu \in \Z_{m_i}}\right)_{i = 1}^\vf\right) 
&\sim \left( \De, (c_n)_{n=1}^\mf, (\epsilon_{\sigma(n)}^{1})_{n=1}^\mf, \left( m_i, [\tilde{\mathsf{s}}^i_\mu, \mathsf{g}^i_{\mu, \nu}]_{\mu, \nu \in \Z_{m_i}}\right)_{i = 1}^\vf\right)\\
&= \left( \De, (c_n)_{n=1}^\mf, (\epsilon_n^{2})_{n=1}^\mf, \left( m_i, [\tilde{\mathsf{s}}^i_\mu, \mathsf{g}^i_{\mu, \nu}]_{\mu, \nu \in \Z_{m_i}}\right)_{i = 1}^\vf\right).
\end{align*}
Thus, $\Phi$ is indeed injective, and therefore bijective, completing the proof.
\end{proof}


\begin{remark}
 In the context of a marked labeled semitoric polygon, the wall-crossing index $w_a^\alpha$ as in Definition~\ref{def:wallcrossing} can be interpreted as the signed count of cuts through the given point. The downwards cuts contribute a -1 and the upwards cuts contribute a +1.
\end{remark}

\section{Application to studying the twisting index invariant} \label{sec:twist}
 
Let $c_i$ be a focus-focus value. 
Considering Equation~\eqref{eq:tilde-to-notilde}, we see that given the semi-local invariant $[\mathsf{s}_\mu^i,\mathsf{g}_{\nu,\mu}^i]_{\nu,\mu\in\Z_{m_i}}$ (as in~\cite{PT}) of the fiber $F^{-1}(c_i)$ this determines $[\tilde{\mathsf{s}}_\mu^i,\mathsf{g}_{\nu,\mu}^i]_{\nu,\mu\in\Z_{m_i}}$ \emph{a priori} up to adding an integer multiple of $2\pi X$ to each action Taylor series $\tilde{\mathsf{s}}_{0},\ldots,\tilde{\mathsf{s}}_{m_{i}-1}$ and translating each of those Taylor series by a constant.
In fact, in this section we will see that given $[\mathsf{s}_\mu^i,\mathsf{g}_{\nu,\mu}^i]_{\nu,\mu\in\Z_{m_i}}$ the set of possible choices of $[\tilde{\mathsf{s}}_\mu^i,\mathsf{g}_{\nu,\mu}^i]_{\nu,\mu\in\Z_{m_i}}$ is naturally isomorphic to $\Z\times\R$, where the discrete freedom corresponds to the choice of twisting index data for the fiber.

We now directly generalize the twisting index and height invariants from~\cite{PVN2009} to the non-simple case. An in-depth study of the twisting index in the case of simple systems is given in~\cite{AHP-twist}.
\begin{definition} \label{def:twist-height}
  Write $c_i = (c_i^1,c_i^2)$, let $(p_\mu^i)_{\mu \in \Z_{m_i}}$ be the tuple of focus-focus points in $F^{-1}(c_i)$, and fix a choice of piecewise affine coordinates $\mathcal{A} = (\mathcal{A}^1,\mathcal{A}^2)$ as in Definition~\ref{def:affinecoords}. 
  We define the \emph{twisting index invariant of $p_\mu^i$ relative to $\mathcal{A}$} by
  \begin{equation} \label{eqn:twist-def}
    \mathsf{k}_\mu^i = \left\lfloor \frac{1}{2\pi} (\tilde{\mathsf{s}}_\mu^i)^{(1,0)} \right\rfloor
  \end{equation}
  and the \emph{height invariant of $c_i$} by
  \begin{equation} \label{eqn:height-def}
    \mathsf{h}^i = \frac{1}{2\pi} (\tilde{\mathsf{s}}_{0}^i)^{(0,0)}-\min_{\ell_{c_i^1}\cap\De} \mathcal{A}^2.
  \end{equation}
\end{definition}

Note that $\mathsf{k}_\mu^i$ depends on the choice of $\mathcal{A}$ but $\mathsf{h}^i$ does not (by Lemma~\ref{lem:Gaction-Taylor}),
and note that $\tilde{\mathsf{s}}_0^i$ can be replaced by any $\tilde{\mathsf{s}}_\mu^i$, $\mu\in\Z_{m_i}$, in Equation~\eqref{eqn:height-def} without changing the value of $\mathsf{h}^i$.
In~\cite{PVN2009} the twisting index is defined by comparing the toric momentum map $\mathcal{A}\circ F$ to a local preferred momentum map, which is nearly equivalent to the definition given in Equation~\eqref{eqn:twist-def} but differs by a shift due to a slight change of convention between~\cite{VN2003} and~\cite{PT}, see Remark~\ref{rmk:twist-convention}.
Let $\Psi\colon \R[[X,Y]]/(2\pi X\Z) \to \R [[X,Y]]$ be the right inverse of the map $\tilde{\mathsf{s}} \mapsto \tilde{\mathsf{s}}-\tilde{\mathsf{s}}^{(0,0)}+2\pi X \Z$ determined by the requirement that $(\Psi (\mathsf{s}))^{(1,0)}\in [0,2\pi)$ for all $\mathsf{s}\in\R [[X,Y]]/(2\pi X \Z)$.
Then
\begin{equation} \label{eqn:twist-tilde1} 
  \tilde{\mathsf{s}}_\mu^i = \Psi(\mathsf{s}_\mu^i) + 2\pi \mathsf{k}_\mu^i X + 2\pi \mathsf{h}^i
\end{equation}
which follows immediately from the definitions of $\mathsf{k}_\mu^i$ and $\mathsf{h}^i$.

The following proposition explains how in some cases only partial data can
be enough to completely recover the semi-local invariants and twisting indices.

\begin{proposition} \label{prop:twistprop}
  Let $c_i$ be a focus-focus value such that $F^{-1}(c_i)$ contains $m_i\in\Z_{>0}$ focus-focus points denoted $(p_\mu^i)_{\mu\in\Z_{m_i}}$. 
  Let $\mathcal{A}$ be a choice of piecewise affine coordinates as in Definition~\ref{def:affinecoords}, and let $\mathsf{k}_\mu^i$ denote the twisting index of $p_\mu^i$ relative to $\mathcal{A}$ as in Definition~\ref{def:twist-height}.
  Let $[\mathsf{s}_\mu^i,\mathsf{g}_{\mu, \nu}^i]_{\mu, \nu \in \Z_m}$ denote the Taylor series invariant as in~\cite{PT}, which can be obtained from the invariant given in Definition~\ref{def:Taylor_series} via Equation~\eqref{eq:tilde-to-notilde}.
  Then $\mathsf{k}^i_0$, $\mathsf{s}^i_0$, and $(\mathsf{g}^i_{\mu, \mu+1})_{\mu \in \Z_{m_i}}$ determine the entire tuple $(\mathsf{k}^i_\mu, \mathsf{s}^i_\mu, \mathsf{g}^i_{\mu, \nu})_{\mu, \nu \in \Z_{m_i}}$.
  That is, given $\mathsf{k}^i_0$, $\mathsf{s}^i_0$, and $(\mathsf{g}^i_{\mu, \mu+1})_{\mu \in \Z_{m_i}}$ there is exactly one possible way to extend
  to $(\mathsf{k}^i_\mu, \mathsf{s}^i_\mu, \mathsf{g}^i_{\mu, \nu})_{\mu, \nu \in \Z_{m_i}}$ under the assumptions above. 
  In particular, the semi-local invariant $(\mathsf{s}_\mu^i,\mathsf{g}_{\mu, \nu}^i)_{\mu, \nu \in \Z_m}$ and the data of a single twisting index $\mathsf{k}_0^i$ determine the other twisting indices $\mathsf{k}_1^i, \ldots, \mathsf{k}_{m-1}^i$.
\end{proposition}

\begin{proof}
  Using Equation~\eqref{eqn:twist-tilde1}, $\mathsf{s}_0^i$ and $\mathsf{k}_0^i$ determine all terms in $\tilde{\mathsf{s}}_0^i$ except the constant term.
Similarly to~\cite{PT}, the relations in Equation~\eqref{eq:invariant-constraint} imply 
that $\tilde{\mathsf{s}}_{0}^i-(\tilde{\mathsf{s}}_{0}^i)^{(0,0)}$, i.e.~the non-constant terms of
$\tilde{\mathsf{s}}_0^i$, and $(\mathsf{g}^i_{\mu,\mu+1})_{\mu\in \Z_{m_i}\setminus\{m_i-1\}}$ completely determine the entire set $(\tilde{\mathsf{s}}_\mu^i-(\tilde{\mathsf{s}}_{\mu}^i)^{(0,0)},\mathsf{g}^i_{\mu,\nu})_{\mu,\nu\in\Z_{m_i}}$. In turn each $(\tilde{\mathsf{s}}_\mu^i)^{(1,0)}$ can be used to determine $\mathsf{k}_\mu^i$ for $\mu \in \Z_m \setminus \{0\}$ by Equation~\eqref{eqn:twist-def} and $\mathsf{s}_\mu^i$ can be recovered 
from the non-constant terms of $\tilde{\mathsf{s}}_{\mu}^i$ by Equation~\eqref{eq:tilde-to-notilde}.
\end{proof}

In Proposition~\ref{prop:twistprop} we use $(\mathsf{s}^i_\mu, \mathsf{g}^i_{\mu, \nu})_{\mu, \nu \in \Z_{m_i}}$
instead of $[\mathsf{s}^i_\mu, \mathsf{g}^i_{\mu, \nu}]_{\mu, \nu \in \Z_{m_i}}$ so we can specify a single twisting
index $\mathsf{k}_0^i$, but the result also implies that $[\mathsf{s}^i_\mu, \mathsf{g}^i_{\mu, \nu}]_{\mu, \nu \in \Z_{m_i}}$
and any single twisting index, assigned to one
of the elements of $[\mathsf{s}^i_\mu, \mathsf{g}^i_{\mu, \nu}]_{\mu, \nu \in \Z_{m_i}}$, determine the remaining twisting indices.

\begin{remark}\label{rmk:twist-convention}
 Alternately, one can take the definition of the twisting index to be given by 
 \[
  \mathsf{k}_{\mu,\textrm{alternate}}^i = \left\lfloor \frac{1}{2\pi} \left( (\tilde{\mathsf{s}}_\mu^i)^{(1,0)} -\frac{\pi}{2}\right) \right\rfloor
 \]
  to agree with the original definition by Pelayo--V\~{u} Ng\d{o}c~\cite[Definition 5.9]{PVN2009} in the case of simple systems with upward cuts, in which case all of the results in this section hold verbatim after redefining $\Psi$ to have the property that $(\Psi(\mathsf{s}))^{(1,0)}\in \left[\frac{\pi}{2},\frac{5\pi}{2} \right)$.
  Again, this difference is due to the shift by $\frac{\pi}{2}X$ in the definition of the Taylor series between~\cite{PT} and~\cite{VN2003}, see~\cite[Remark 3.10]{PT}.
\end{remark}

\begin{remark}
  Proposition~\ref{prop:twistprop} and Equation~\eqref{eqn:twist-tilde1} imply that for each focus-focus fiber $F^{-1}(c_i)$ the data encoded in $[\tilde{\mathsf{s}}_\mu,\mathsf{g}_{\mu,\nu}]_{\mu,\nu\in\Z_{m_i}}$ is equivalent to the data of the semi-local invariant $[\mathsf{s}_\mu,\mathsf{g}_{\mu,\nu}]_{\mu,\nu\in\Z_{m_i}}$ from~\cite{PT}, the height invariant $\mathsf{h}^i$ from Equation~\eqref{eqn:height-def}, and additionally \emph{one integer}, the twisting index $\mathsf{k}_0^i$ of any one of the focus-focus points in the fiber.
\end{remark}

\section{Example} \label{sec:examples}

An explicit example of a compact semitoric system which includes a twice-pinched torus can be obtained by certain choices of parameters for the system described in Hohloch--Palmer~\cite{HohPal}, which is a generalization of the coupled angular momentum system, see for instance~\cite{LFP2018, ADH}.

\begin{figure}[ht]
  \centering
  \begin{subfigure}[t]{.4\linewidth}
    \centering
    \inputfigure{HP-example-single}
    \caption{If $s_1 = \frac{1}{2} + \varepsilon$ for small $\varepsilon>0$ there are two once-pinched fibers.}
    \label{fig:HPsingle}
  \end{subfigure}\quad
  \begin{subfigure}[t]{.4\linewidth}
    \centering
    \inputfigure{HP-example-double}
    \caption{If $s_1 = \frac{1}{2}$ there is one twice-pinched fiber.}
    \label{fig:HPdouble}
  \end{subfigure}
  \caption{The polygon in a representative of the complete semitoric invariant associated to the system given in Equation~\eqref{eq:HPsystem} for
   $s_1$ close to $\frac{1}{2}$ and $s_1=\frac{1}{2}$. The wall-crossing indices as in Definition~\ref{def:wallcrossing} are indicated to the left
   of the line segments they label, and shown above the marked points are the corresponding focus-focus fibers.
   The arrows indicate that the focus-focus fibers are mapped to the marked points by $\mathcal{A}\circ F$, where $\mathcal{A}$ is the choice
   of piecewise affine coordinates (as in Definition~\ref{def:affinecoords}) associated to the polygon in the figure.}
  \label{fig:HP-example}
\end{figure}

Consider $M = \mathbb{S}^2 \times \mathbb{S}^2$ with coordinates $(x_1, y_1, z_1, x_2, y_2, z_2)$ inherited from the usual inclusion $\mathbb{S}^2 \subset \R^3$ and symplectic form $\om = -(R_1 \om_{\mathbb{S}^2} \oplus R_2 \omega_{\mathbb{S}^2})$ for some parameters $R_1, R_2>0$, where $\om_{\mathbb{S}^2}$ is the usual area form on the sphere giving it area $2\pi$.
The parameters $R_1, R_2$ represent the radii of the spheres and in~\cite{HohPal} it was assumed that $R_1 < R_2$ since taking $R_1 = R_2$ can produce non-simple semitoric systems or even systems which include degenerate singularities (depending on the other parameters).
Using parameters $R_1 = R_2 = 1$ and $s_2 = s_1$, the semitoric system given in~\cite[Theorem 1.2]{HohPal} is 
$
 (\mathbb{S}^2 \times \mathbb{S}^2, -(\om_{\mathbb{S}^2} \oplus \omega_{\mathbb{S}^2}), F_{s_1} = (J, H_{s_1}))
$
where
\begin{equation} \label{eq:HPsystem}
\begin{cases}
  J = z_1 + z_2,\\
  H_{s_1} = (1-s_1)^2 z_1 + s_1^2 z_2 + 2s_1 (1-s_1)(x_1 y_1 + x_2 y_2),
\end{cases}
\end{equation}
and $s_1 \in [0, 1]$ remains a free parameter.
For $s_1$ in a neighborhood of $\nicefrac{1}{2}$ this is a semitoric system with two focus-focus points which occur at $(0, 0, \pm 1, 0, 0, \mp 1)$.
If $s_1 = \nicefrac{1}{2}$ then both of these points are in the same fiber $F^{-1}(0, 0)$, which is thus a twice-pinched torus.
Taking $s_1 = \nicefrac{1}{2} + \varepsilon$, with $\varepsilon\neq 0$ sufficiently small, produces a semitoric system which has two focus-focus points which are in different fibers of $F$, but nevertheless both focus-focus points lie in the same fiber of $J$, so it still does not satisfy the simplicity condition.
Figure~\ref{fig:HP-example} shows the polygon and focus-focus fibers in each case.
The semitoric polygon can be determined from the single representative shown in the figure, which has vertices at $(-2,-1)$, $(0,1)$, $(2,1)$, and $(0,-1)$ in both cases and has wall-crossing indices labeled. Computing
the other invariants is more difficult, see for instance the techniques used in~\cite{ADH, ADHspin}.
It would be interesting to compare the Taylor series invariants for the two separate focus-focus fibers in Figure~\ref{fig:HPsingle} to
the Taylor series invariants for the twice-pinched focus-focus fiber in Figure~\ref{fig:HPdouble}.

\section{Final remarks and open problems}
\label{sec:final}

\subsection{Group actions and moduli problems}
One can think of a symplectic $4$-manifold as the phase space of a
mechanical system, and an action or integrable system on it as an
additional symmetry. 
Therefore, one may view Theorems~\ref{thm:classification} and~\ref{thm:equiv-classification} as
a symplectic classifications of Hamiltonian $(\mathbb{S}^1\times\R)$-actions on
symplectic $4$-manifolds (under the constraints on the types
of singularities which can occur in semitoric systems).
The symplectic classification of Hamiltonian $(\R\times\R)$-actions
is expected to be difficult and essentially corresponds to classifying
integrable systems with two degrees of freedom on symplectic $4$-manifolds. The one degree of freedom case appeared
in~\cite{DuMoTo1994}.
On the other hand, there also exists a symplectic classification
of symplectic $(\mathbb{S}^1\times \mathbb{S}^1)$-actions on $4$-manifolds~\cite{Pelayo-symplecticaction},
and Hamiltonian $\mathbb{S}^1$-actions on compact $4$-manifolds have also been classified~\cite{karshon}. For the case of Hamiltonian completely
integrable torus actions on orbifolds, see~\cite{LerTol1997}
which extends the manifold case from~\cite{De88}.
The existence of these symmetries can have implications on the
topology and geometry of $M$. For instance,
Karshon~\cite{karshon} proved that if $M$ is a compact symplectic $4$-manifold
which admits a Hamiltonian $\mathbb{S}^1$-action then $M$ is K\"ahler (and
hence there are strong constraints on the topology of $M$, 
such as the fundamental group).
In fact, the proof of Delzant's theorem~\cite{De88} shows that
if $M$ is a $2n$-dimensional symplectic manifold which admits a Hamiltonian $n$-torus action then $M$ is a toric variety (and toric
varieties are simply connected, see~\cite{Danilov78} and also
see~\cite{DuPe2009} for an explicit construction of the toric
variety with charts). Recently, there has been interest on extending
Delzant's classification to log-symplectic manifolds~\cite{GuMiPi2014,GuLiPeRa2017}.

Finally, one natural follow up problem to the classification
of this paper is to apply the result to study the structure of the
moduli space that semitoric systems form.
In the case of simple semitoric systems this was done
in~\cite{Pal-moduli2017} and in the case of toric systems in dimension $4$ in 
Pelayo--Pires--Ratiu--Sabatini~\cite{PePiRaSa2014} and later in arbitrary dimension in 
Pelayo--Santos~\cite{PeSa}.
Having fibers with multiple
focus-focus points per fiber makes the problem even more
interesting, as one needs to face the question of deformations
between fibers which include different numbers of
focus-focus points.
The articles~\cite{LFPfamilies,LFPfamilies2} contain a number of results
concerning the behavior of families under variations of parameters.

\subsection{Inverse spectral problems for simple and non-simple systems}\label{sec:quantum}
The inverse spectral problem for integrable systems asks how much information may be obtained from the semiclassical spectrum of the system, in~\cite{Pel22} the second author proposes some related open problems about classical and quantum integrable systems.
Theorem~\ref{thm:classification} includes the data of the twisting index inside of a Taylor series, giving a possible
 strategy to detect the  twisting index from the semiclassical spectrum of a semitoric system. 

It was conjectured by Pelayo--V\~{u} Ng\d{o}c~\cite[Conjecture 9.1]{PVN2011} that from the semiclassical spectrum of a quantum simple semitoric system one can recover the associated simple semitoric system. In view of the classification of simple semitoric systems~\cite{PVN2009,PVN2011}  this amounts to proving that the five invariants of simple semitoric systems can be recovered from the semiclassical spectrum. 

The first four invariants (number of focus-focus values, the Taylor series associated with each focus-focus value,  the polygonal invariant of the system, the height invariant associated with each focus-focus value) can be recovered from the semiclassical spectrum by Pelayo--V\~{u} Ng\d{o}c~\cite{PVN14} and Le Floch--Pelayo--V\~{u} Ng\d{o}c~\cite[Theorem A]{LFPVN2016}, which implies that the conjecture holds for simple semitoric systems within the same twisting-index class and for simple semitoric systems with exactly one focus-focus point (\emph{Jaynes-Cummings systems}) and normalized twisting index at that point (for a detailed study of the Hamiltonian dynamics of the example of simple semitoric system known as coupled-spin oscillator see~\cite{PeVNspin2012,ADHspin} and for their spectral theory see~\cite{PeVNspin2012}).
By Le Floch--V\~{u} Ng\d{o}c~\cite{LFVN} the last invariant may also be recovered from the semiclassical spectrum, which implies the remaining case of the conjecture. In the same paper the authors give algorithms to recover the invariants starting from a semiclassical spectrum which is known to be the semiclassical spectrum of some quantum semitoric system. We refer to~\cite[Section 10]{Pel22} for an expository account. In~\cite[Section 12.3]{Pel22} the construction/surjectivity problem for quantum semitoric systems is proposed, in the spirit of the construction/surjectivity statement of Delzant's classification of compact toric systems (and of the classification of semitoric systems).


A proof of the conjecture for toric systems on compact manifolds of any dimension $2n$ was given by Charles, Pelayo and V\~{u} Ng\d{o}c in~\cite{ChPeVN2013} using in an essential way the Atiyah--Guillemim--Sternberg~\cite{At82, GS82} and Delzant~\cite{De88} results which show that the only invariant of such systems in the image of the system (a convex polytope in $\R^n$).

\subsection{Recovering the system from the affine structure}\label{sec:recover-from-affine}
It is a classical problem in integrable systems to ask under what conditions an integrable system can be recovered from the affine structure on the base of the associated Lagrangian fibration. 
This is related to the question of when a classical integrable system can be recovered from the joint spectra of the associated quantum integrable system (see Section~\ref{sec:quantum} and the references therein). 
In both cases in principle there is no reason to believe that such a recovery is always possible, but to the authors' knowledge
the following question is open: is it possible for two non-isomorphic non-simple semitoric systems which have a multiply-pinched focus-focus fiber to have the same affine manifold as their base or have the same joint spectra of the associated quantum integrable system?
We believe that the answer to this question is ``yes'', and it seems to be the consensus among experts that this question can be answered in the affirmative, but we could not find any specific references in the literature.

\paragraph{Acknowledgements.} 
J.~Palmer thanks Rutgers University for supporting travel related to this paper, and the AMS and the Simons Foundation which also funded a visit to the second author through an AMS-Simons travel grant.
He was also partially supported by the \emph{FWO senior postdoctoral fellowship} 12ZW320N.
 Research on this paper was carried out by \'A. Pelayo at UC San Diego, and later at the Complutense University of Madrid with support of BBVA Foundation Grant for Scientific Research Projects with project title ``From Integrability to Randomness in Symplectic and Quantum Geometry (FITRISAQG)''.
X.~Tang thanks Cornell University, University of Toronto, and Tsinghua University where this paper was completed and thanks the Beijing Institute of Technology for a Research Fund Program for Young Scholars.
  
\bibliographystyle{alpha}
\bibliography{ref}

\bigskip

\authaddresses

\end{document}